\documentclass[a4paper,final,onefignum,onetabnum]{siamart190516} 
\usepackage[utf8]{inputenc}
\usepackage[active]{srcltx}
\usepackage{amsfonts,amsmath,amssymb,bm,graphicx}
\usepackage{xcolor}
\usepackage[utf8]{inputenc}
\usepackage{multirow}
\usepackage{array}
\usepackage{enumitem}
\usepackage{type1cm}
\usepackage{bbm}
\usepackage{verbatim}
\usepackage{hyperref}

\newcommand{\RR}{\mathbb{R}}
\newcommand{\ZZ}{\mathbb{Z}}
\newcommand{\NN}{\mathbb{N}}
\newcommand{\TT}{\mathbb{T}}

\title{Adaptive solution of initial value problems by a dynamical Galerkin scheme\footnotemark[1]\ \footnotemark[2]}

\author{R.M. Pereira \footnotemark[3] \and  N. Nguyen van yen \footnotemark[4]\ \footnotemark[5] \and K. Schneider \footnotemark[6]\ \and M. Farge \footnotemark[4]} 

\begin{document}
\maketitle

\renewcommand{\thefootnote}{\fnsymbol{footnote}}
\footnotetext[1]{
The authors would like to thank Greg Hammett for a discussion  
which strongly motivated this work.
The French Federation for Fusion Studies and the PEPS program of CNRS-INSMI are acknowledged for financial support.
}
\footnotetext[2]{
This work, 
supported by the European Communities under the contract of Association between EURATOM, 
CEA and the French Research Federation for Fusion Studies, 
was carried out within the framework of the European Fusion Development Agreement. 
The views and opinions expressed herein do not necessarily reflect those of the European Commission.
}
\footnotetext[3]{Instituto de F\'isica, Universidade Federal Fluminense, Niterói, Brazil}
%
\footnotetext[4]{LMD--CNRS, \'Ecole Normale Sup\'erieure-PSL, Paris, France}
\footnotetext[5]{FB Mathematik und Informatik, Freie Universit\"at Berlin, Berlin, Germany. NVY thanks the Humboldt foundation for post-doctoral support.}
\footnotetext[6]{Institut de Math\'ematiques de Marseille (I2M), CNRS, Aix-Marseille Universit\'e, Marseille, France}

\begin{abstract}
We study dynamical Galerkin schemes for evolutionary partial differential equations (PDEs), where the projection operator changes over time. 
When selecting a subset of basis functions, the projection operator is non-differentiable in time and an integral formulation has to be used. 
We analyze the projected equations with respect to existence and uniqueness of the solution and prove that non-smooth projection operators introduce dissipation, a result which is crucial for adaptive discretizations of PDEs, e.g., adaptive wavelet methods.
For the Burgers equation we illustrate numerically that thresholding the wavelet coefficients, and thus changing the projection space, will indeed introduce dissipation of energy.
We discuss consequences for the so-called `pseudo-adaptive' simulations, where time evolution 
and dealiasing are done in Fourier space, whilst thresholding is carried out in wavelet space.
Numerical examples are given for the inviscid Burgers equation in 1D and the incompressible Euler equations in 2D and 3D. 
\end{abstract}




\pagestyle{myheadings}
\thispagestyle{plain}
\markboth{R.M. Pereira, N. Nguyen van yen, K. Schneider and M. Farge}{Adaptive solution of initial value problems by a dynamical Galerkin scheme}

\section{Introduction}
Motivated by high accuracy at reduced computational cost with respect to uniform grid methods,
numerous adaptive discretization schemes of evolutionary partial differential equations (PDEs) have been developed since decades, see, e.g., \cite{Brandt77}.
Real world problems, for instance, fluid and plasma turbulence, or reactive flows, typically involve a multitude of active spatial and temporal scales and adaptivity allows to concentrate the computational effort at locations and time instants where it is necessary to ensure a given numerical accuracy, while elsewhere efforts may be significantly reduced. 
Among adaptive approaches, multiresolution and wavelet methods offer an attractive possibility to introduce locally refined grids, which dynamically track the evolution of the solution in space and scale. Automatic error control of the adaptive discretization, with respect to a uniform grid solution, is hereby an advantageous feature \cite{Cohen00}. For a review of adaptive multiresolution methods in the context of computational fluid dynamics (CFD) we refer to \cite{ScVa10}.

In many applications, in particular in CFD, Galerkin truncated discretizations of the underlying PDEs which use a finite number of modes are the methods of choice.
Spectral methods \cite{CQHZ88} are a prominent example and Fourier-Galerkin schemes are widely used for direct numerical simulation of turbulence \cite{IGKa2009} due to their high accuracy. For efficiency reasons the convolution product in spectral space, due to the nonlinear quadratic term and typically encountered in hydrodynamic equations, is evaluated in physical space and aliasing errors are completely removed.
This implementation, called pseudo-spectral formulation with full dealiasing using the $2/3$ rule, is equivalent to a Fourier-Galerkin scheme up to round-off errors \cite{CQHZ88}. Thus the discretization conserves the $L^2$-norm of the solution.
A classical test to check the stability of pseudo-spectral codes for viscous Burgers or Navier-Stokes equations is to perform simulations with vanishing viscosity. This allows to verify if the $L^2$ norm of the solution, i.e., typically energy, is conserved and for sufficiently small time steps the truncated Galerkin schemes are stable.
However, the solution of the Galerkin truncated inviscid equations, e.g., inviscid Burgers or incompressible Euler, shows artefacts in the form of oscillations and the computed solution is not physical.
Already T.D. Lee \cite{Lee1952} predicted energy equipartition between all Fourier coefficients in spectral approximations for 3D incompressible Euler, called thermalization, by applying Liouville's theorem from statistical mechanics.

The effect of truncating Fourier-Galerkin schemes has been studied in \cite{RFNM11,MFNBR2020} for the 1D Burgers and 2D incompressible Euler equations. The observed short-wavelength oscillations were named `tygers' and were interpreted as first manifestations of thermalization \cite{Lee1952}. The proposed cause was the resonant interaction between fluid particle motion and truncation waves.

Motivated by this work, detailed numerical analysis of Fourier-Galerkin methods for nonlinear evolutionary PDEs, in particular for inviscid Burger and incompressible Euler, was then performed in \cite{BaTa13}.
The authors showed spectral convergence for smooth solutions of the inviscid Burgers equation and the  incompressible Euler equations. 
However, when the solution lacks sufficient smoothness, then both the spectral and the $2/3$ pseudo-spectral Fourier methods exhibit nonlinear instabilities which generate spurious oscillations.
In particular it was shown that after the shock formation in the inviscid Burgers equation, the total variation of bounded (pseudo-) spectral Fourier solutions must increase with the number of increasing modes. 
The $L^2$-energy conservation of the spectral solution is reflected through spurious oscillations, which is in contrast with energy dissipating Onsager solutions.
A complete explanation of these nonlinear instabilities was thus given and `tygers' \cite{RFNM11} were demystified.

To remove these non-physical oscillations in Galerkin truncated approximations different numerical regularization techniques have been proposed, commonly used in numerical methods for solving hyperbolic conservation laws.
If the solution is not unique the `regularized' numerical scheme selects one weak solution, which should correspond to the physically relevant one, e.g., the entropy solution of the inviscid Burgers equation, which can be computed exactly using the Legendre transform \cite{Vergassola1994}.
These approaches include upwind techniques \cite{Osher1982}, total variation diminishing schemes \cite{Harten1983}, shock limiters \cite{Sweby1984}, spectral vanishing viscosity \cite{Tadmor1989, Gottlieb2001}, classical viscosity and hyperviscosity \cite{BLSB1981} and also inviscid regularization schemes \cite{BLTi2008, KhTi2008}.


In the context of adaptive wavelet schemes, numerical experiments with the 1D inviscid Burgers equation showed that wavelet filtering of the Fourier-Galerkin truncated solution in each time step, which corresponds to denoising and is removing the oscillations, yields the solution to the viscous Burgers equation \cite{Nguyenvanyen2008}.
%
%
For the 2D  incompressible  Euler  equations \cite{Nguyenvanyen2009} different wavelet techniques for regularizing truncated Fourier-Galerkin solutions were studied using either real-valued or complex-valued wavelets and the results were compared with viscous and hyperviscous regularization methods.
The results show that nonlinear wavelet filtering with complex-valued wavelets preserves the flow dynamics and suggest $L^2$ convergence to the reference solution. The wavelet representation offers at the same time a non negligible compression rate of about $3$ for fully developed 2D turbulence.

Simulations of the 3D wavelet-filtered Navier-Stokes equations \cite{OYSFK11} showed that statistical predictability of isotropic turbulence can be preserved with a reduced number of degrees of freedom. 
This approach, called Coherent Vorticity Simulation (CVS) \cite{FSK1999} is a multiscale method to compute incompressible turbulent flows based on the wavelet filtered vorticity field.
The coherent vorticity, corresponding to the few coefficients whose modulus is larger than a threshold, represents the organized and energetic flow part, while the remaining incoherent vorticity is noise like.
Applying wavelet-based denoising, i.e., CVS filtering, to the 3D Galerkin truncated incompressible Euler equations confirmed that this adaptive regularization models turbulent dissipation and thus allows to compute turbulent flows with intermittent nonlinear dynamics and a $k^{-5/3}$ Kolmogorov energy spectrum \cite{Farge2017}.
A significant compression rate of the wavelet coefficients of vorticity is likewise observed which reduces the number of active degrees of freedom to only about 3.5\% of the total number of coefficients for the studied turbulent flows, computed at Taylor microscale based Reynolds number of $200$.

Filtering the wavelet representation of the Galerkin truncated inviscid Burgers and 2D incompressible Euler equations in \cite{PNFS13}, by retaining only the significant coefficients, showed that the spurious oscillations due to resonance can be filtered out, and dissipation can thus be introduced by the adaptive representation.

The aim of the current work is to provide a rigorous mathematical framework to analyze and to understand the properties of adaptive discretizations of evolutionary PDEs based on dynamical Galerkin schemes. To this end we analyze these adaptive Galerkin discretizations.
Galerkin schemes by itself are particularly appealing due to their optimality properties, conservation of energy and the ease of numerical analysis using Hilbert space techniques. Introducing space adaptivity, e.g., by wavelet filtering in each time step, implies that the projection operator changes over time as only a subset of basis functions is used. 
Hence, the projection operator is non-differentiable in time and we propose the use of an integral formulation. 
The projected equations are then analyzed with respect to existence and uniqueness of the solution.
It is proven that non-smooth projection operators introduce dissipation, a result which is crucial for adaptive discretizations of nonlinear PDEs.
Existence and uniqueness of the solution of the projected equations is likewise shown.
Tools from countable systems of ordinary differential equations and functional analysis in Banach spaces are used. For related background we refer the reader to text books \cite{Deim77, Schwabik1992} and \cite{Fili2013}.

The remainder of the article is organized as follows.
Dynamical Galerkin schemes are defined in section~\ref{sec:dyngal} and the existence and uniqueness of the projected equations is analyzed giving an explanation of the introduced energy dissipation.
Space and time discretization of the Burgers and incompressible Euler equations is described in section~\ref{sec:discret}.
Numerical examples are presented in section~\ref{sec:numex} to illustrate the dissipation mechanism.
Section~\ref{sec:appl} shows applications of the CVS filtering to the inviscid Burgers equation in 1D and the 2D and 3D incompressible Euler equations.
Some conclusions are drawn in section~\ref{sec:concl}.


\section{Dynamical Galerkin schemes}
\label{sec:dyngal}

  \subsection{Motivation}

Evolutionary PDEs can be discretized with a Galerkin method in space,
by projecting the equation onto a sequence of finite dimensional linear spaces, which approximate the solution in space
when the discretization parameter, $h$, goes to zero.
Using truncation to a finite number of modes, the infinite dimensional countable system of ordinary differential equations in time can be reduced.
An important restriction of such methods is that the projection space typically does not evolve in time and the number of modes is fixed. Here, we propose a formulation of adaptive Galerkin discretizations where the projection operator and the number of modes can change over time and we show that under suitable conditions adaptation can introduce dissipation.

  \subsection{Formal definition}

Let $H$ be a Banach space, and consider the evolution equation
\begin{equation}
\label{eq:burgers_abstract}
u' = f(u) 
\end{equation}
where $u'$ denotes the weak time derivative of $u$
and $f$ is defined and continuous from some sub-Banach space $D(f) \subset H$ into $H$. 
Equation (\ref{eq:burgers_abstract}) is completed by a suitable initial condition $u(0)= u(t=0)$.
To be more specific, we shall focus below on the case of the one-dimensional Burgers equation on the torus $\RR/\ZZ$:
\begin{equation}
\label{eq:burgers}
\partial_t u + u \partial_x u = \nu \partial_{xx} u
\end{equation}
which corresponds to (\ref{eq:burgers_abstract}) with
\begin{equation}
\label{eq:burgers_case}
f(u) = \nu \partial_{xx} u - u \partial_x u
\end{equation}
and $u = u(x,t)$.

The classical Galerkin discretization of (\ref{eq:burgers_abstract}) is defined as follows: 
for $h>0$, let $H_h$ be a fixed finite dimensional subspace of $D(f)$,
such that:
$$
\overline{\bigcup_{h>0} H_h} = H
$$
where the adherence is taken in $H$, 
and let $P_h$ be the orthogonal projector on $H_h$.
Find $u_h : [0,T] \in H_h$ such that:
\begin{equation}
\label{eq:burgers_galerkin}
u_h' = P_h f(u_h) = P_h(\nu \partial_{xx} u_h - u_h \partial_x u_h)
\end{equation}

Now for $t\in [0,T]$, assume that $P_h(t)$ is an orthogonal projector on some 
finite dimensional subspace $H_h(t)$ of $H$.
The dimension of $H_h(t)$ is allowed to change in time,
but we assume that $H_h(t)$ remains within a fixed finite dimensional subspace $H_h^0$.
$P_h$ therefore takes its values in the set of orthogonal projectors $H_h^0 \to H_h^0$,
which we denote by $\Pi_h^0$, with its natural smooth manifold structure as a closed subset of all linear mappings $H_h^0 \to H_h^0$.
We want to find $u_h : [0,T] \in H_h(t)$ which is an approximation of $u$.

Let us first assume that $P_h$ is a smooth function of time.
As in the case where $P_h$ is time independent, we apply $P_h(t)$ to the differential equation to get:
\begin{equation}
\label{eq:burgers_dynamic_galerkin_1}
P_h(t) u_h'(t) = P_h(t) f(u_h(t))
\end{equation}
but now, since $P_h$ does not commute with the time-derivative,
this equation is not sufficient to determine $u_h'(t)$ entirely.
We need another equation to fix the component of $u_h'(t)$ which is in the orthogonal of $H_h(t)$, i.e., in $H^\perp_h(t)$.

To derive this equation, we start from the condition that $u_h(t) \in H_h(t)$ for every $t$,
which is equivalent to
\begin{equation}
P_h(t) u_h(t) = u_h(t).
\end{equation}
Differentiating in time this identity leads to:
\begin{equation}
P_h(t) u_h'(t) + P_h'(t) u_h = u_h'(t)
\end{equation}
or equivalently
\begin{equation}
\label{eq:burgers_dynamic_galerkin_2}
\left(1-P_h(t) \right) u_h'(t) = P_h'(t) u_h (t)
\end{equation}
which is exactly the equation we were looking for.
By adding (\ref{eq:burgers_dynamic_galerkin_1}) and (\ref{eq:burgers_dynamic_galerkin_2}) together,
we obtain the definition of the dynamical Galerkin scheme:
\begin{equation}
\label{eq:burgers_dynamic_galerkin}
u_h'(t) = P_h(t) f \left(u_h(t)\right) + P_h'(t) u_h(t)
\end{equation}

By comparing this differential equation with (\ref{eq:burgers_galerkin}), 
we observe the appearance of a new term proportional to the time-derivative of $P_h$.
This is the essential ingredient which characterizes the dynamical Galerkin scheme.
We now show the following 

\medskip

\begin{lemma}
\label{thm:smooth_projector}
Any solution of (\ref{eq:burgers_dynamic_galerkin}) such that $u_h(0) \in H_h(0)$ also satisfies $u_h(t) \in H_h(t)$ for all $t$,
and moreover
\begin{equation}
\label{eq:energy_equation}
\frac{1}{2} \frac{\mathrm{d}}{\mathrm{d}t} \Vert u_h(t) \Vert^2 = (u_h(t), f(u_h(t)))
\end{equation}
\end{lemma}

\begin{proof}

By differentiating $P_h(t)^2 = P_h(t)$ and $P_h(t)^3 = P_h(t)$ respectively,
we obtain the identities $$P_h(t) P_h(t)' + P_h(t)' P_h(t) = P_h(t)' \quad \mathrm{and} \quad P_h(t) P_h(t)' P_h(t) = 0,$$
which imply that
\begin{equation}
\frac{\mathrm{d}}{\mathrm{d}t}((1-P_h(t))u_h(t)) = 0
\end{equation}
and the first part follows.
To prove the second part, take the inner product of the equation with $u_h$:
\begin{equation}
\frac{1}{2} \frac{\mathrm{d}}{\mathrm{d}t} \Vert u_h(t) \Vert^2 = (u_h(t), f(u_h(t))) + (u_h(t), P_h'(t) u_h(t)),
\label{eq:energy_eq2}
\end{equation}
where the last term can be rewritten 
\begin{equation}
(P_h(t) u_h(t), P_h'(t) P_h(t) u_h(t)) = (u_h(t), P_h(t) P_h'(t) P_h(t) u_h(t)) = 0 \; ,
\nonumber
\end{equation}
which proves (\ref{eq:energy_equation}).
\end{proof}

\medskip

The above computations are valid when $P_h$ is differentiable,
which is a severe restriction and forbids us in particular to switch on and off dynamically some functions in the basis of integration,
which is the goal that we had set ourselves in the beginning.
To pursue we therefore need to extend the definition of the scheme to non-differentiable $P_h$.
For this we consider the integral formulation of (\ref{eq:burgers_dynamic_galerkin}), namely
\begin{equation}
\label{eq:burgers_dynamic_galerkin_integral1}
u_h(t) = u_h(0) + \int_0^t P_h(\tau) f(u_h(\tau)) \mathrm{d}\tau +  \int_0^t P_h'(\tau) u_h(\tau)\mathrm{d}\tau.
\end{equation}
This equation can be rewritten using a Stieltjes integral with respect to $P_h$:
\begin{equation}
\label{eq:burgers_dynamic_galerkin_integral}
u_h(t) = u_h(0) + \int_0^t P_h(\tau) f(u_h(\tau)) \mathrm{d}\tau +  \int_0^t \mathrm{d} P_h(\tau) u_h(\tau)
\end{equation}
which we call the integral formulation of the dynamical Galerkin scheme.

This equation makes sense as soon as $P_h$ has bounded variation (BV),
which gives it a much wider range of applicability than (\ref{eq:burgers_dynamic_galerkin}),
allowing in particular discontinuities in $P_h$.
To solve such an equation we need to resort to the theory of generalized ordinary differential equations,
which we now recall.

  \subsection{Existence and uniqueness of a solution to the projected equations}

The rigorous setting for integral equations such as (\ref{eq:burgers_dynamic_galerkin_integral}) involving Stieltjes integrals
is explained in detail in the book \cite{Schwabik1992}.
An alternative introduction can be found in \cite{Pandit1982}.
We summarize the main consequences of the theory for our problem in the following:

\medskip

\begin{theorem}
Assume that $P_h(t):[0,T] \to $ is BV and left-continuous, that $P_h(0) u_h(0) = u_h(0)$ (i.e., $u_h(0) \in H_h(0)$),
and that $f:H_h^0 \to H$ is locally Lipschitz.
Then
\begin{enumerate}[label=(\roman{*}), ref=(\roman{*})]
\item[(i)]
There exists $T^*$, $0 < T^* \leq T$, such that the integral equation 
\begin{equation}
\label{eq:burgers_dynamic_galerkin_integral0}
u_h(t) = u_h(0) + \int_0^t P_h(\tau) f(u_h(\tau)) \mathrm{d}\tau +  \int_0^t \mathrm{d} P_h(\tau) u_h(\tau)
\end{equation}
has a unique BV, left-continuous solution $u_h : [0,T^*] \to H_h^0$.
\item[(ii)]
This solution satisfies
\begin{equation}
\forall t \in [0,T], P_h(t) u_h(t) = u_h(t)
\end{equation} 
\item[(iii)]
$u_h$ is continuous at any point of continuity of $P_h$,
and more generally for any $t$:
\begin{equation}
u_h(t^+) - u_h(t) = ( P_h(t^+) - P_h(t) ) u_h(t)
\end{equation}
or equivalently
\begin{equation}
u_h(t^+) = P_h(t^+) u_h(t)
\end{equation}
\item[(iv)]
The energy equation (\ref{eq:energy_equation}) for smooth $P_h$ is replaced in general by:
\begin{multline}
\label{eq:energy_equation2}
\frac{1}{2} (\Vert u_h(t) \Vert^2-\Vert u_h(0) \Vert^2) = \\ \int_0^t (u_h(\tau), f(u_h(\tau)))\mathrm{d}\tau - \frac{1}{2} \sum_{\{i\mid t_i < t\}} \Vert (1-P_h(t_i^+)) u_h(t_i) \Vert^2,
\end{multline}
where $(t_i)_{i \in \NN}$ are the points of discontinuity of $P_h$.
\end{enumerate}
\end{theorem}

\medskip

\begin{proof}
To prove part (i) of the theorem we first need to familiarize ourselves with a few key concepts used by \cite{Schwabik1992}.
\begin{definition}
Let $G = \{ x \in \RR^n \mid \Vert x \Vert \leq c \} \times [0,T]$,
$h : [0,T] \to \RR$ a non decreasing, continuous from the left function,
and $\omega : [0,+\infty) \to \RR$ a continuous, increasing function with $\omega(0) = 0$.

We will say that a function $F : G \to \RR^n$ belongs to the class $\mathcal{F}(G,h,\omega)$, 
if and only if
\begin{equation}
\Vert F(x,t_2) - F(x,t_1) \Vert \leq \vert h(t_2) - h(t_1) \vert
\end{equation}
and
\begin{equation}
\Vert F(x,t_2) - F(x,t_1) - F(y,t_2) + F(y,t_1) \Vert \leq \omega(\Vert x - y \Vert) \vert h(t_2) - h(t_1) \vert
\end{equation}
for all $(x,t_2), (x,t_1), (y,t_2), (y,t_1) \in G$.
\end{definition}

\medskip

The proof of the existence is based on the Schauder-Tichonov fixed point theorem, using theorem 4.2, p. 114 of ref.~\cite{Schwabik1992}.
The uniqueness can be shown using theorem 4.8, page 122 of ref.~\cite{Schwabik1992} proving the local uniqueness property in the future, i.e., for increasing $t$.

Now let us turn to (ii).
The idea is to approximate $P_h$ by a family of smooth functions $P_{h,\varepsilon}$, $\varepsilon > 0$,
and then to apply Lemma \ref{thm:smooth_projector} to the corresponding solution $u_{h,\varepsilon}$,
giving
\begin{equation}
\left(1-P_{h,\varepsilon}(t) \right) \, u_{h,\varepsilon}(t) = 0
\end{equation}
and then passing to the limit.
For this we need $u_{h,\varepsilon}(t) \to u_{h}(t)$,
which means that the solution depends continuously on $P_h$ (see chapter 8 p. 262 : continuous dependence on parameters).

The continuity of $u_h$ in part (iii) follows directly from the fact that $P_h$ is left-continuous and BV. 

The energy equation in part (iv) can be shown by integrating (\ref{eq:energy_eq2}) in time and replacing $P_h'(t) u_h(t)$ by $(1 - P_h(t)) u'_h(t)$, cf. (\ref{eq:burgers_dynamic_galerkin_2}).
\end{proof}


In the case when the projector $P_h(t)$ depends on $u(t)$, e.g., when using adaptive wavelet thresholding, we have,

\begin{subequations}
\label{eq:coupled_dynamical_galerkin}
\begin{align}
\label{eq:coupled_dynamical_galerkin1}
u_h(t) & = u_h(0) + \int_0^t P_h(\tau) f(u_h(\tau)) \mathrm{d}\tau +  \int_0^t \mathrm{d} P_h(\tau) u_h(\tau) \\
\label{eq:coupled_dynamical_galerkin2}
P_h(t) & = \Phi(u_h(t))
\end{align}
\end{subequations}

\begin{theorem}
Under certain conditions, the system (\ref{eq:coupled_dynamical_galerkin}) has a unique solution.
\end{theorem}

\begin{proof}
We proceed by iteration.
Let $P_h^0$ be the projector on the time-independent approximation space $H_h^0$,
$u_h^0$ be the corresponding solution of (\ref{eq:coupled_dynamical_galerkin1}).
We then define recursively 
\begin{equation}
P_h^{n+1}(t) = \Phi(u_h^n(t))
\end{equation}
and $u_h^{n+1}$ as the solution of (\ref{eq:coupled_dynamical_galerkin1}) with $P_h = P_h^{n+1}$.
\end{proof}

  \section{Space and time discretization}
  \label{sec:discret}
For space discretization in the numerical results below we use a classical Fourier pseudo-spectral scheme \cite{CQHZ88}.
The spectral Fourier projection of $u \in L^1(\TT^d)$ where $\TT = \RR / (2 \pi \ZZ)$ is given by
\begin{equation}
P_N u (\bm x) = u_N (\bm x) = \sum_{|{\bm k}| \lesssim N/2} \widehat u_k \, e^{i {\bm k} \cdot {\bm x}} \; , \; \widehat u_{\bm k} = \frac{1}{(2 \pi)^d} \int_{\TT^d}\, u({\bm x}) \, e^{-i {\bm k}  \cdot {\bm x}} \, d{\bm x}
\label{eq:Fourierprojector}
\end{equation}
Note that $|{k}| \lesssim N/2$ is understood in the sense $-N/2 \le k < N/2$ and correspondingly in higher dimensions for each component of $\bm k$.

Applying the spectral discretization to the one-dimensional inviscid Burgers equation ($d=1$),
\begin{equation}
\partial_t u + \frac{1}{2} \partial_x u^2 \, = \, 0  \quad {\text{for}} \quad x \in \TT \quad {\text{and }} \quad t>0 
\label{eq:inviscidBurgers}
\end{equation}
with periodic boundary conditions and suitable initial condition $u(x,t=0) = u_0(x)$ yields the Galerkin scheme
\begin{equation}
\partial_t u_N + \frac{1}{2} \partial_x \left( P_N (u_N)^2 \right) \, = \, 0  \quad {\text{for}} \quad x \in \TT \quad {\text{and }} \quad t>0 
\end{equation}
which corresponds to a nonlinear system of $N$ coupled ODEs for $\widehat u_k(t)$ with $|{k}| \lesssim N/2$.
A pseudo-spectral evaluation of the nonlinear term is utilized, and the product in physical space is fully dealiased. In other words, the Fourier modes retained in the expansion of the solution are such that $|k| \le k_C$, where $k_C$ is the desired cut-off wave number, but the grid has $N= 3k_C$ points in each direction, versus $N= 2k_C$ for a non-dealiased, critically sampled product. This dealiasing makes the pseudo-spectral scheme equivalent to a Fourier-Galerkin scheme up to round-off errors \cite{CQHZ88}, and is thus conservative. 

For the two- and three-dimensional incompressible Euler equations ($d=2, 3$) with periodic boundary conditions, 
\begin{eqnarray}
\label{eq:Euler}
\partial_t {\bm u} + \left( {\bm u} \cdot \nabla \right) {\bm u}  \, & = & \, - \nabla p  \quad {\text{for}} \quad {\bm x} \in \TT^d \quad {\text{and }} \quad t>0  \\ \nonumber
\nabla \cdot {\bm u} & = & 0
\end{eqnarray}
a similar spectral discretization can be applied. The pressure $p$ is eliminated using the Leray projection onto divergence free vector fields. Eventually a nonlinear system of coupled ODEs is obtained for the Fourier coefficients of the velocity $\widehat{\bm u}_{\bm k}(t)$.

For time discretization of the resulting ODE systems we stick to classical Runge-Kutta schemes, of order 4 for the 1D Burgers equation and the 3D Euler equations, while for 2D Euler 3rd order Runge-Kutta with a low storage formulation is used, see \cite{Orlandi2000}, on page 20.
For details on the convergence and stability of the above spectral schemes we refer to \cite{BaTa13}.
Implementation features for the 1D Burgers equation and the 2D Euler equation can be found in \cite{Nguyenvanyen2009} and \cite{PNFS13}.
For details on the scheme for 3D Euler we refer to \cite{Farge2017}.
\begin{figure}
\begin{center}
\includegraphics[width=0.9\textwidth]{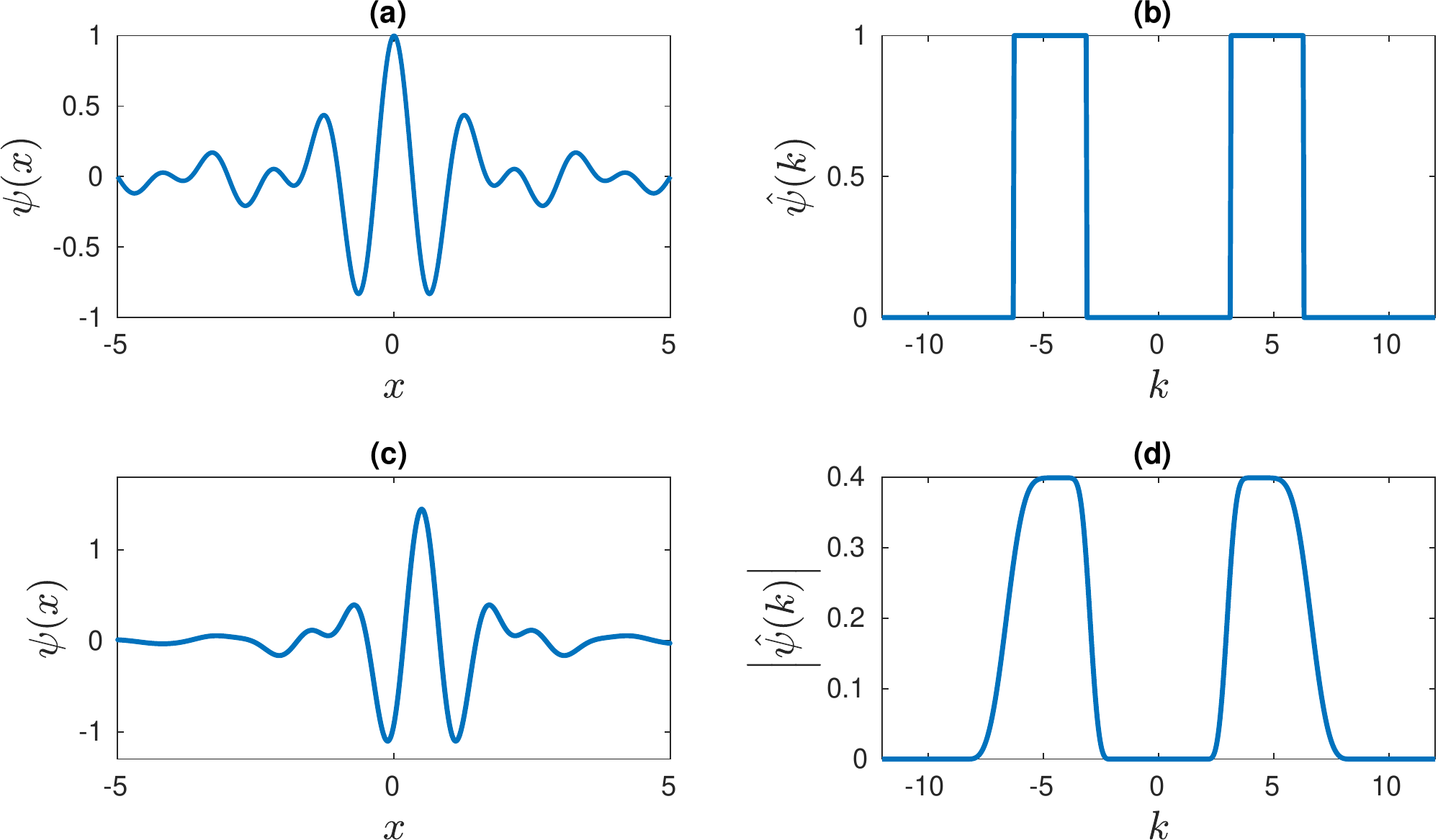}
\end{center}
\caption{Shannon wavelet (top) and Meyer wavelet (bottom) in physical space $\psi(x)$ (left) and the corresponding modulus of the Fourier transform $|\widehat{\psi}(k)|$ (right).\label{fig:wavelet}}
\end{figure}
%

The Fourier space discretization described above could be replaced by any other Galerkin discretization, using for instance finite elements, or wavelets as basis functions. The interest of using wavelets is to introduce adaptive discretizations, see e.g., \cite{ScVa10, ESRF2021}. In this case the projector $P$ is changing over time and is non smooth, which means that dissipation is introduced by removing/adding basis functions during the time stepping. This technique has been previously used for regularizing the Burgers equation and the incompressible Euler equations without a rigorous mathematical justification.

To test the influence of wavelet thresholding we introduce the concept of pseudo-adaptive simulations. The Fourier Galerkin discretization is used to solve the PDE, but in each time step the numerical solution $u_N$ is decomposed into a periodic orthogonal wavelet series  of $L^2(\TT^d)$. For $d=1$ we thus have the 1D truncated wavelet series 
\begin{equation}
P_J u_N(x) = u^J_N(x) = \overline u_{00} + \sum_{j=0}^{J-1} \sum_{i=0}^{2^j -1} \widetilde u_{ji} \psi_{ji} (x)  \, , \quad \widetilde u_{ji} = \int_{\TT} u_N(x) \psi_{ji} (x) dx
\label{eq:OWS}
\end{equation}
where $\overline u_{00}$ is the mean value of the solution and $\widetilde u_{ji}$ its wavelet coefficients. The wavelet $\psi_{ji}(x) = 2^{j/2} \psi(2^j x - i)$ quantifies fluctuations at scale $2^{-j}$ around position $i/2^j$  and $N=2^J$ denotes the total number of grid points, corresponding to the finest resolution.
Figure~\ref{fig:wavelet} illustrates Shannon and Meyer wavelets together with the corresponding Fourier transforms, which have compact support. This implies that both are trigonometric polynomials and can be spanned by a Fourier basis.
For extensions to higher dimensions using tensor product constructions of wavelets, we refer to the literature \cite{daubechies1992}.

Wavelet filtering, which is the basis of the Coherent Vorticity Simulation (CVS) \cite{FSK1999}, introduces a sparse representation of the solution, by removing weak wavelet coefficients.
Thresholding of the wavelet coefficients with a threshold $\epsilon$, which typically depends on time, is performed.
This yields a projection of the numerical solution $u_N$
\begin{equation}
P_J^{\epsilon} u_N(x) = u^J_{\epsilon} (x) = \overline u_{00} + \sum_{j=0}^{J-1} \sum_{i=0}^{2^j -1} \rho_{\epsilon}\left(\widetilde u_{ji} \right) \psi_{ji} (x)  \, ,
\label{eq:OWS_epsilon}
\end{equation}
where $\rho_{\epsilon}$ is the (hard) thresholding operator defined as,
\begin{equation}
\rho_{\epsilon}( x) \, = \,   \left \{
    \begin{array} {ll}
        x \quad \quad \quad \; \mbox{\rm for} \quad  |x| > \epsilon\\
        0 \quad \quad \quad \; \mbox{\rm for} \quad  |x| \le \epsilon\\
    \end{array}
  \right.
\label{eqn:hardthres}
\end{equation}
and $\epsilon$ denotes the threshold.
The thresholding error can be estimated (see e.g., \cite{Cohen00}) and we have $$ || P_J u_N(x) - P_J^{\epsilon} u_N(x) ||_2 \le C \epsilon \, . $$

\medskip

Using pseudo-adaptive simulations the CVS algorithm can be summarized as follows \cite{PNFS13}:

\medskip

\begin{itemize}
    \item[i)] The Fourier coefficients of the solution $\widehat u_k$ for $|{k}| \lesssim N/2$ are advanced in time to $t=t_{n+1}$ and an inverse Fourier transform is applied on a grid of size $N$ to obtain $u_N$.
    \item[ii)] A forward wavelet transform is performed to obtain $P_J u_N(x)$, according to equation (\ref{eq:OWS}).
    \item[iii)] CVS filtering removes wavelet coefficients having magnitude below the threshold $\epsilon$. The threshold value is determined iteratively~\cite{azzalini2005} and initialized with $\epsilon_0 = q \sqrt{||u||_2 / 2 / N}$ where $q$ is a compression parameter. The iteration steps are then obtained by $\epsilon_{s+1} = q \sigma[\widetilde u^{s}_{ji}]$ until $\epsilon_{s+1} = \epsilon_s$, where $\widetilde u^{s}_{ji}$ are the wavelet coefficients below  $\epsilon_s$ and $\sigma[\cdot]$ is the standard deviation of the set of these coefficients.
    \item[iv)] A safety zone is added in wavelet space. The index set of retained wavelet coefficients in step iii) is denoted by $\Lambda$ and for each retained wavelet coefficient indexed by $(j,i) \in \Lambda$ neighboring coefficients in position and scale (5 in the present case) are added, as illustrated in figure~\ref{fig:safetyzone}.
    \item[v)] An inverse wavelet transform is applied to the wavelet coefficients above the final threshold and a Fourier transform is then performed to obtain the Fourier coefficients of the filtered solution at time step $t_{n+1}$.
\end{itemize}

\medskip

Different choices of the wavelet basis for regularization have been tested, e.g., in \cite{PNFS13}, including various orthogonal wavelets and a Dual-Tree Complex Wavelet basis we refer to as `Kingslets' \cite{Kingslets}. 
The value of the compression parameter $q$ controls the number of discarded coefficients and in previous studies we found experimentally
the value $q=5$ for `Kingslets' (complex-valued wavelets) and for orthogonal wavelets we used $q=8$.

Adding a safety zone is necessary due to the lack of translational invariance of orthogonal wavelets, but also for local dealiasing. The idea is to keep neighboring coefficients in space and scale and to account for translation of shocks or step gradients and the generation of finer scale structures. For complex-valued wavelets, which are translation invariant, no safety zone is required, as shown in~\cite{PNFS13}.
For details and further discussion on possible choices of the safety zone we refer the reader to~\cite{OYSFK11}.

\begin{figure}
\begin{center}
\includegraphics[width=0.55\textwidth]{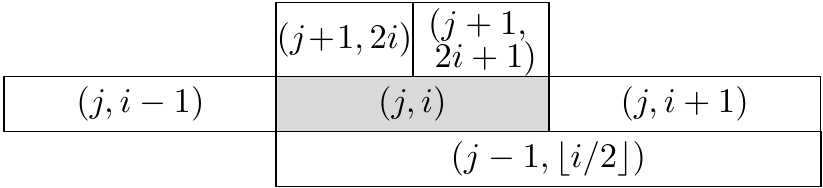}
\end{center}
\caption{Safety zone in wavelet coefficient space around an active coefficient $(j, i)$ in position $i$ and finer ($j+1$) and coarser scale ($j-1$).\label{fig:safetyzone}}
\end{figure}


\section{Numerical experiments} 
\label{sec:numex}

In the following we show results to illustrate the properties of dynamical Galerkin scheme and in particular their ability to introduce energy dissipation into the numerical method, which can be useful for stabilization.
As examples we consider first the inviscid 1D Burgers equation using  periodic boundary conditions. The initial condition is
a simple sine wave given by $u(x,t=0) = \sin(2 \pi x)$ for $x \in \TT$. Unless explicitly noted, computations are done with $N=2048$ collocation points and the time step $\Delta t$ is chosen so that $\Delta x/\Delta t = 16$, where $\Delta x = 1/N$ is the grid discretization size. This choice ensures the CFL condition is met \cite{CQHZ88}.


\begin{figure}
\begin{center}
\includegraphics[width=0.9\textwidth]{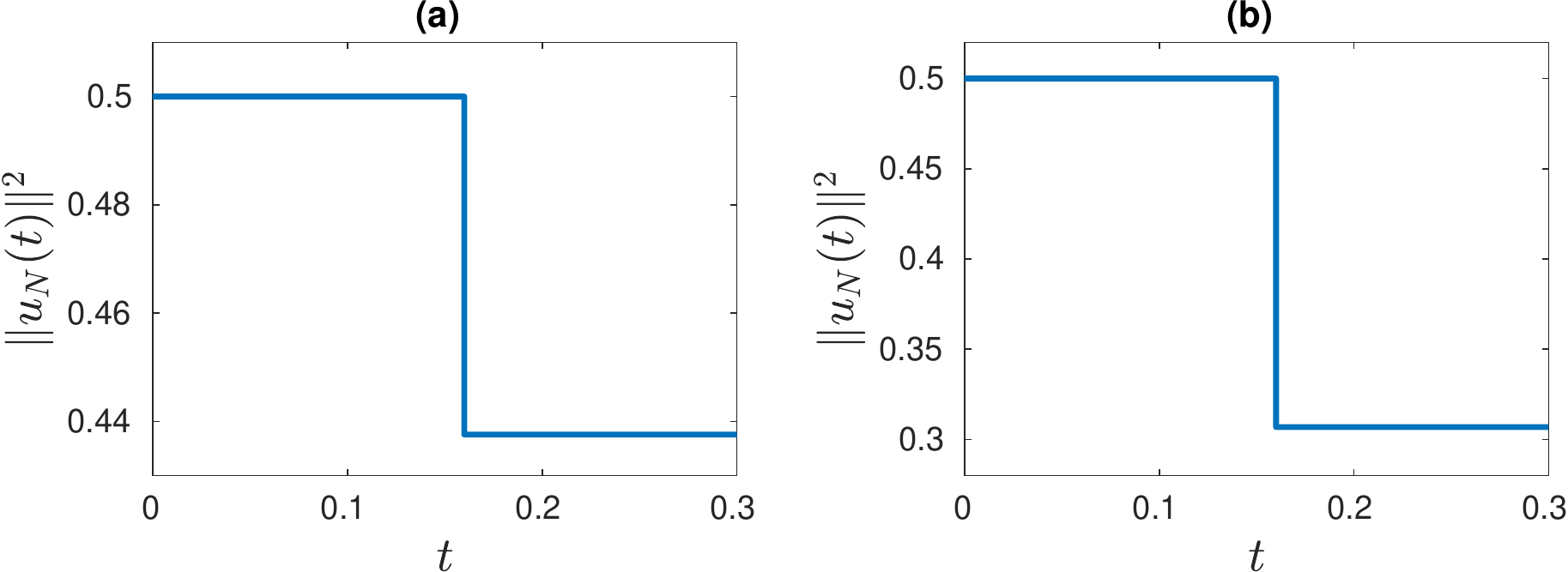}
\end{center}
\caption{Filtering of one mode in (a) Fourier space and (b) in wavelet space for the inviscid 1D Burgers equation. Time evolution of energy. As expected, energy loss is observed.\label{fig:test}}
\end{figure}
%

\begin{figure}
\begin{center}
\includegraphics[width=0.56\textwidth]{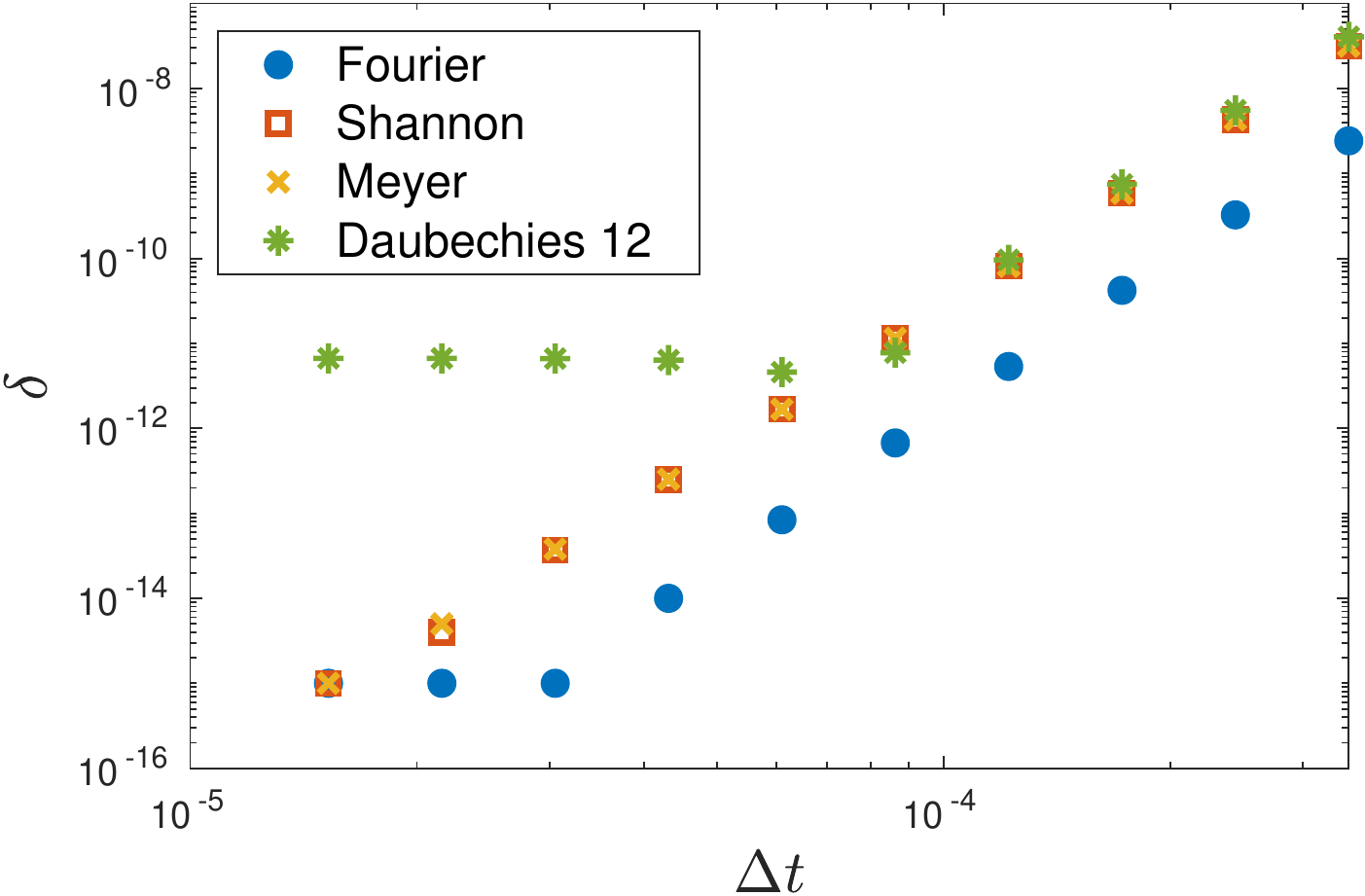}
\end{center}
\caption{Difference between dissipated energy and filtered energy (equation \ref{eq:delta}) as a function of the time step $\Delta t$, when a single Fourier mode or wavelet coefficient is filtered. A residual difference remains when Daubechies wavelets are employed. \label{fig:test2}}
\end{figure}
%

  \subsection{Punctual selection in the Fourier basis} 

The simplest illustration which we develop as a proof of concept is a punctual selection in the Fourier basis. Starting at some time instant $t_b$ and during an entire interval $[t_b, t_e]$, we set to zero the Fourier coefficients corresponding to a given wave number $k_f$ after each time step (both positive and negative modes are erased, such that the solution remains real). 
The projection operator thus becomes time dependent and discontinuous and we have
\begin{equation}
P_N(t)_{[t_b, t_e]}^{k_f} u (x) \, = \,   \left \{
    \begin{array} {ll}
        \sum_{|{k}| \lesssim N/2, |k| \ne k_f} \widehat u_k \, e^{i { k} \, {x}}\quad \quad \; \mbox{\rm for} \quad  {t} \in [t_b, t_e] \\
        \sum_{|{ k}| \lesssim N/2} \widehat u_k \, e^{i { k} \, { x}} \quad \; \, \, \quad \quad \quad  \; \mbox{elsewhere.}\\
    \end{array}
  \right.
\label{eqn:P_filter_fourier}
\end{equation}
The removal of these modes will instantly dissipate energy of the numerical solution, but from there on energy is conserved. And this is the case still after the reintroduction of the coefficients in the projection basis, despite the discontinuity of the projection operator. Indeed, according to (\ref{eq:energy_equation2}) dissipation is observed as long as $\Vert (1-P_h(t^+)) u_h(t) \Vert^2$ is non zero, but at $t=t_e$ this quantity is null and therefore energy is conserved. We note that since a 
multistage time marching scheme is employed, it is necessary to reset to zero the removed coefficients after each substage, to ensure they have no effect on the solution.

We show in figure~\ref{fig:test}(a) the time evolution of the energy when the filtering wave number is $k_f=2$. The projection operator changes at $t_b=0.16$ and is then restored at $t_e=0.2 $. Dissipation is introduced by this change of projection basis and, up to numerical errors, the lost energy amounts to the energy content of the discarded coefficients. This can be seen in figure~\ref{fig:test2}, where we plot, as a function of the time step $\Delta t$, the quantity
\begin{equation}\label{eq:delta}
\delta = (\Vert u_N(0) \Vert^2-\Vert u_N(t_b) \Vert^2) - \Vert (1-P_N(t_b^+)^{k_f}_{[t_b,t_e]}) u_N(t_b)\Vert^2,
\end{equation}
which should be zero according to (\ref{eq:energy_equation2}), since the PDE is energy conserving up to time $t_b$. One observes that $\delta$ indeed converges to zero up to machine precision (of order 10$^{-15}$) as $\Delta t$ is decreased.

  \subsection{Punctual selection in real orthogonal wavelet bases}

To illustrate dissipation through reprojection on a wavelet basis, we extend the previous idea of a punctual selection now to wavelet space. The solution of the Fourier Galerkin method is decomposed in each time step into an orthogonal wavelet basis, as in equation (\ref{eq:OWS}). One single energy containing coefficient, of scale index $j_f$ and position index $i_f$, is then set to zero after every time step during some given time interval $[t_b,t_e]$. The projection operator is once again time dependent and discontinuous, and may be written as
\begin{equation}
P_J(t)_{[t_b, t_e]}^{j_f,i_f} u (x) \, = \,   \left \{
    \begin{array} {ll}
        \overline u_{00} + \sum_{j=0}^{J-1} \sum_{i=0}^{2^j -1} \widetilde u_{ji} \psi_{ji} (x) (1-\delta_{jj_f}\delta_{ii_f})\quad \; \; \mbox{\rm for} \quad  {t} \in [t_b, t_e] \\
        \overline u_{00} + \sum_{j=0}^{J-1} \sum_{i=0}^{2^j -1} \widetilde u_{ji} \psi_{ji} (x) \quad \quad \quad \quad \quad \quad \quad \; \mbox{elsewhere,}\\
    \end{array}
  \right.
\label{eqn:P_filter_wavelet}
\end{equation}
for a chosen orthogonal wavelet $\psi_{ji}(x)$.

We show in figure~\ref{fig:test}(b) the energy time evolution for the case of projections in the Meyer wavelet basis. The filtered coefficient corresponds to $j_f=1$ and $i_f=1$. As before, the filtering happens from time $t_b=0.16$ to $t_e=0.2$. Energy is punctually dissipated as of the first change in the projector, but is otherwise conserved. Figure~\ref{fig:test2} also shows the convergence of the quantity $\delta$ from equation \ref{eq:delta}, now with the projector replaced by equation \ref{eqn:P_filter_wavelet}. Similar results are also obtained with projections onto a Shannon wavelet basis. 

Interestingly, the same convergence is not observed in figure~\ref{fig:test2} when Daubechies wavelets are used. As illustrated in figure~\ref{fig:wavelet}, working with Shannon wavelets is actually equivalent to working with the Fourier basis, since it is compactly supported in spectral space, with a sharp cut-off. Combining multiscale Shannon wavelets amounts to covering the spectral space up to some Galerkin cut-off frequency. When projecting with this basis, one is simply damping some existing Fourier coefficients without introducing new wave numbers. Hence, when going back to the fully dealiased Fourier space, no further energy is lost. The Meyer wavelet is likewise compactly supported in spectral space, however the projection onto Meyer wavelets is only equivalent to a Fourier projection when the number of Fourier modes is increased from $N$ to $3/2 N$, which is the case when dealiasing is applied. 
Therefore, in both cases the dissipated energy indeed corresponds to the energy lost due to the discontinuity of the projection operator. The Daubechies wavelet, on the other hand, is not compactly supported in spectral space. When a projection is made in wavelet space and some coefficient is discarded, this will affect wave numbers beyond the dealiased ones, which then cease to vanish. After returning to Fourier space, the dealiasing operation will set all these to zero and further energy dissipation occurs. For this reason, the quantity $\delta$ shows a residual value as the time step decreases and does not attain machine precision, as seen in figure~\ref{fig:test2}. In this simulation, Daubechies 12 wavelets were employed and the projector corresponds to equation \ref{eqn:P_filter_wavelet} with $j_f=0$ and $i_f=0$. Note that the indices are chosen so that the amount of dissipated energy is comparable in all cases.

This additional energy dissipation can once again be understood as due to a change in the projector, i.e., going from the wavelet projector removing one coefficient, given in equation \ref{eqn:P_filter_wavelet}, to the Fourier projector given in equation \ref{eq:Fourierprojector}.
In other words, it is the fact that these two projectors do not commute when Daubechies wavelets are used (or any other basis not compactly supported in Fourier space, i.e., within the fully dealiased spectral space) which leads to more dissipation then that introduced by the filtering. This shows that pseudo-adaptive simulations, such as those discussed in section~\ref{sec:discret}, must be taken with care, since they may not exactly reproduce what one would get with a fully adaptive scheme in wavelet space. Still, they are valuable tools to predict the solutions behavior in a simpler and faster setup, and we shall apply them to illustrate the introduction of dissipation in conservation laws through a dynamical Galerkin scheme.

\section{Application to the inviscid Burgers equation and incompressible Euler using CVS filtering}
\label{sec:appl}

In the following section we present in a concise way some results from the literature to illustrate the dissipation properties of adaptive Galerkin methods using CVS filtering. We show some numerical examples for the one dimensional inviscid Burgers equation including some space-time convergence and for the incompressible Euler equations in two and three dimensions. For details on the numerical simulations we refer to \cite{PNFS13} and \cite{Farge2017}.

  \subsection{Inviscid Burgers}

We consider the inviscid Burgers equation (\ref{eq:inviscidBurgers}), discretized with a
Fourier pseudo-spectral method and endowed with CVS filtering, described in section~\ref{sec:discret}, using $N=16384$ Fourier modes.
For the used sinusoidal initial condition $u(x,t=0)=\sin(2\pi t)$ the
time evolution of the reference solution, so-called entropy solution, can be easily computed with the method of characteristics, separately in each half of the domain.
Figure \ref{fig:burgers_cvs} shows the solution of the standard Fourier Galerkin method, which preserves energy, and the solution obtained with the dynamic Galerkin scheme using CVS filtering with `Kingslets'. We observe that the oscillations (also called resonances, see \cite{RFNM11}), which appear as soon as the shock is formed, are removed using CVS filtering.
This is further confirmed in figure \ref{fig:burgers_cvs_zoom} (left) where the oscillations are shown to be completely filtered out and a smooth solution close to the reference solution 
is obtained.

To assess the filtering performance, we develop a space-time convergence analysis by computing the time integrated relative $L^2$-distance from the filtered solution $u_N$ to the analytical reference solution $u_\mathrm{ref}$. We compute,
\begin{equation}\label{eq:error}
    \mathcal{E} = \int_{t_0}^{t_1} \frac{\Vert u_N(t) - u_\mathrm{ref}(t)\Vert^2}{\Vert u_\mathrm{ref}(t)\Vert^2} dt,
\end{equation}
for different space resolutions while keeping fixed the previous relation between time and space discretization, that is, $\Delta x/\Delta t = 16$. 
Since the filtering is only relevant after the shock formation, we actually start the analysis from a time right before the shock time $t_s = \inf_x \left[ -1/u'(x,0)\right] \approx 0.1592$, i.e., $t_0 = t_s - \Delta t$ and carry on the integration up to $t_1=0.3$. Results for complex-valued Kingslets and real-valued Shannon wavelets with and without the safety zone discussed in section \ref{sec:discret} are shown in figure \ref{fig:convergence}. 
We can observe that CVS with Kingslets is in excellent agreement with the reference solution, showing an $\mathcal{O}(\Delta x)$ convergence rate. Although typically one order of magnitude poorer (an under-performance that we now quantify but which has only been visually verified in \cite{PNFS13}), CVS with Shannon wavelets also shows first order convergence towards the reference solution if the safety zone is present. Meanwhile, as anticipated in section \ref{sec:discret}, figure \ref{fig:convergence}(c) shows that CVS is not able to properly regularize the solution when employing real orthogonal wavelets if a safety zone is not introduced.

\begin{figure}
\begin{center}
\includegraphics[width=0.96\textwidth]{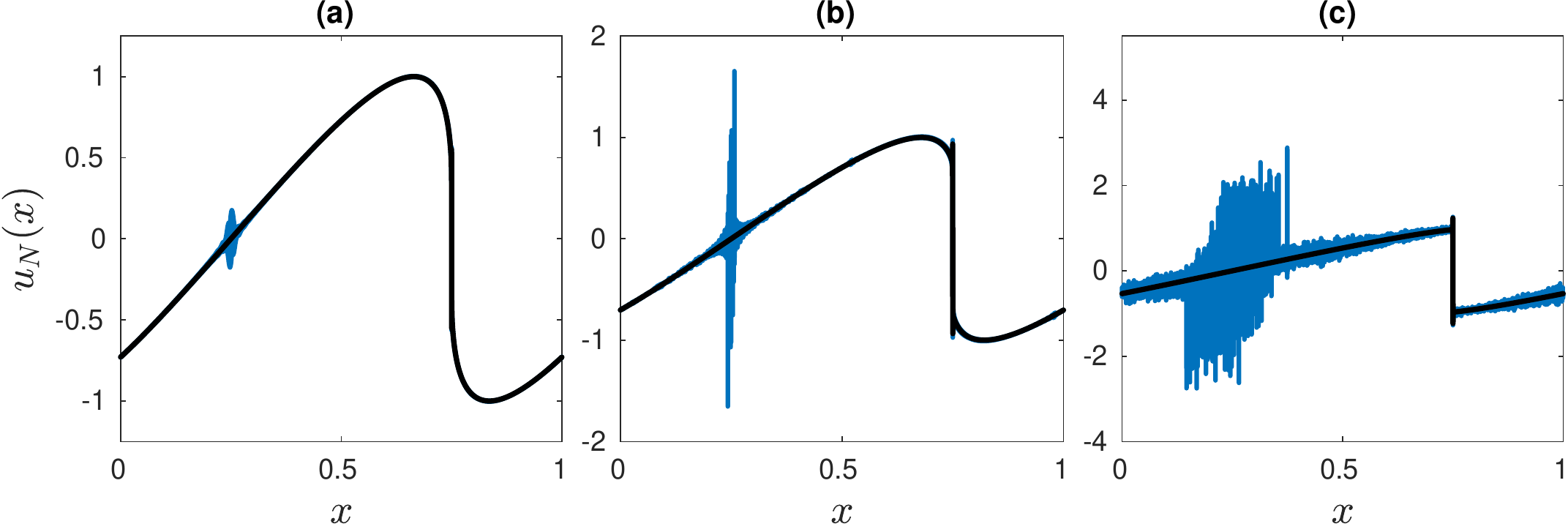}
\end{center}
\caption{CVS-filtered Galerkin truncated inviscid Burgers equation using complex-valued wavelets (Kingslets, in black) together with the non-dissipative Galerkin truncated solution (blue) at times
$t=0.1644$, $0.1793$ and $0.3$. The solutions are periodically shifted to the right, so that both the resonances and the shocks can be easily seen.\label{fig:burgers_cvs}}
\end{figure}

\begin{figure}
\begin{center}
\includegraphics[width=0.98\textwidth]{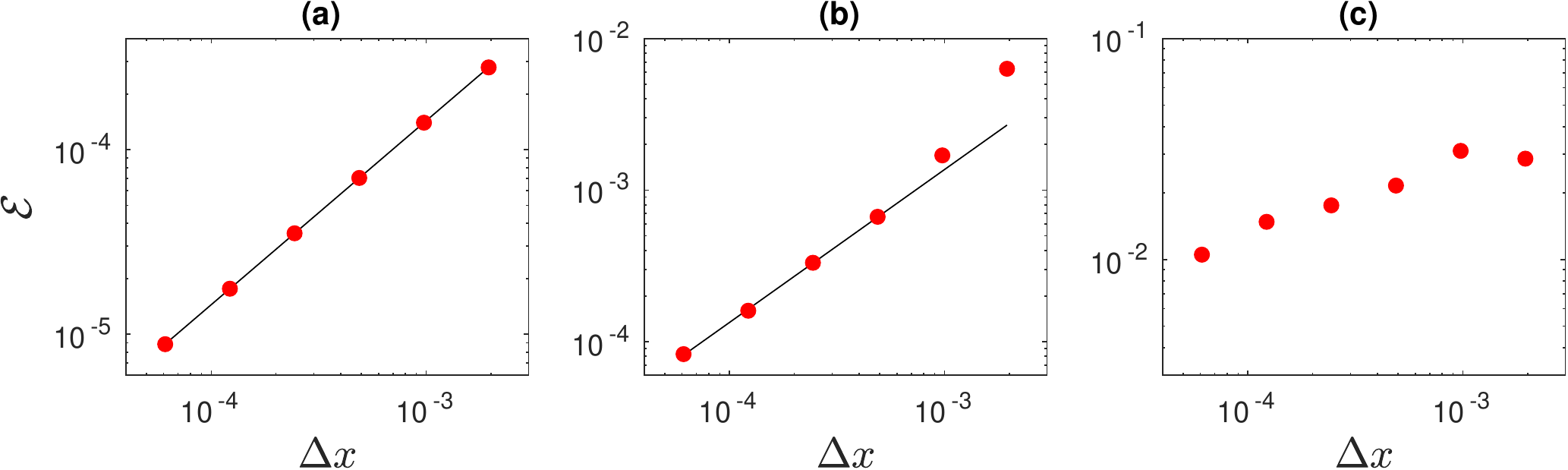}
\end{center}
\caption{Time integrated relative $L^2$-error (equation \ref{eq:error}) as a function of space resolution $\Delta x$. (a) Kingslets (b) Shannon wavelet with the safety zone (c) Shannon wavelet without the safety zone. The straight lines have slope 1.\label{fig:convergence}}
\end{figure}

\begin{figure}
\begin{center}
\includegraphics[width=0.98\textwidth]{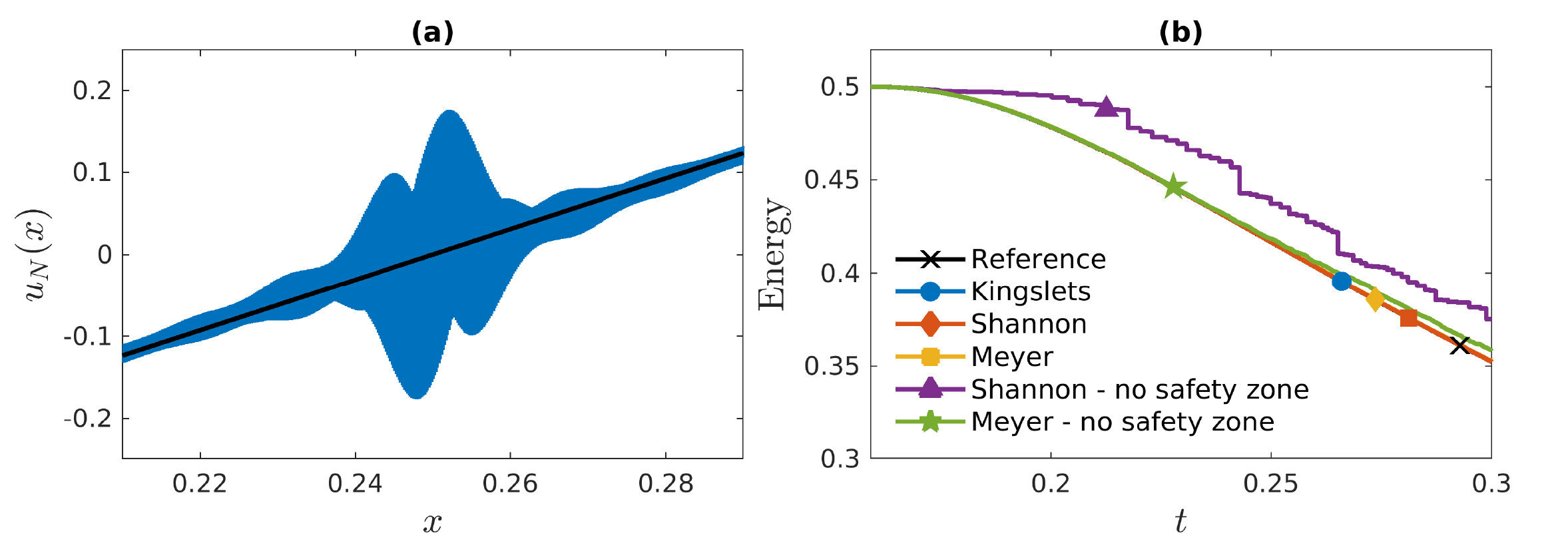}
\end{center}
\caption{(a) Detail of the solution of CVS-filtered Galerkin truncated inviscid Burgers equation using complex-valued wavelets (Kingslets, in black) together with the non-dissipative Galerkin truncated solution (blue) at time $t=0.1644$.
Right: Time evolution of the energy $E(t)$ of CVS filtered solutions for different wavelets with and without safety zone together with the analytical result. 
\label{fig:burgers_cvs_zoom}
}
\end{figure}

The evolution of the energy $E= ||u||^2$ shown in figure~\ref{fig:burgers_cvs_zoom} (right) further quantifies the dissipation of the adaptive schemes for different real orthogonal wavelets. Once again, in the presence of the safety zone the wavelet adaptation removes sufficient energy, matching thus the analytical energy evolution. However, it is now seen that without the safety zone not enough energy is dissipated and the solution is not properly regularized.
For a detailed description of similar simulations and a physical interpretation we refer to \cite{PNFS13}.

  \subsection{Incompressible Euler equations}

To illustrate the effect of dissipation when adapting the basis functions using projectors changing over time we consider the incompressible Euler equations given in (\ref{eq:Euler}) and discretize them with a classical Fourier Galerkin scheme.
In these pseudo-adaptive simulations we apply in each time step CVS filtering.
Detailed results can be found in \cite{PNFS13} and \cite{Farge2017} for the two and three-dimensional cases, respectively.

\begin{figure}[h!]
\begin{center}
\includegraphics[width=1.0\textwidth]{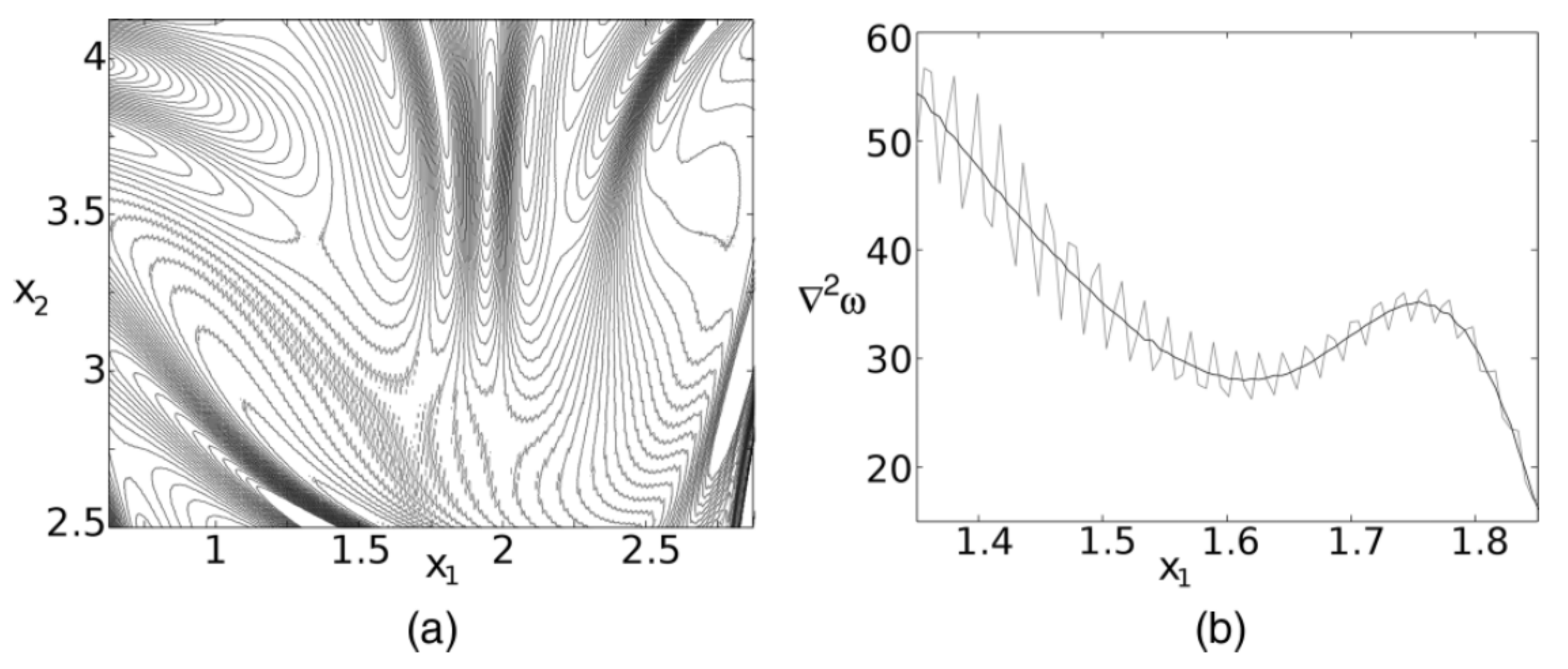}
\end{center}
\caption{Filtering of 2D incompressible Euler using complex-valued wavelets (Kingslets). Left: Contours of the Laplacian of vorticity $\Delta \omega$ at $t=0.71$. The Galerkin truncated solution is shown in gray, the CVS solution is given in black. Right: 1D cut of the Laplacian of vorticity for the oscillatory Galerkin truncated solution and the wavelet-filtered smooth solution. From \cite{PNFS13}.\label{fig:2deulerzoom}}
\end{figure}
In the two-dimensional case a random initial condition is evolved in time with third order Runge-Kutta time integration using a resolution of $N=1024^2$ Fourier modes \cite{PNFS13}.
Visualizations of the Laplacian of vorticity $\omega = \nabla \times {\bm u}$ in the fully developed nonlinear regime are shown in figure~\ref{fig:2deulerzoom} (left). For the Galerkin truncated solution we find oscillations in the isolines in $\Delta \omega$ (a small scale quantity, which is sensitive to oscillations) while the regularized solution using complex-valued wavelets with CVS filtering yields a smooth solution. A one-dimensional cut in figure~\ref{fig:2deulerzoom} (right) illustrates that in the CVS solution the oscillations have been indeed removed. Time evolution of enstrophy, defined as $\frac{1}{2} ||\omega ||_2^2$, shows that in contrast to the Galerkin truncated simulation the CVS computation is dissipative and the enstrophy departs from the one of the conservative Galerkin truncated case and it decays for times larger than 1.4. For more details including a physical interpretation we refer to \cite{PNFS13}.

\begin{figure}
\begin{center}
\includegraphics[width=0.5\textwidth]{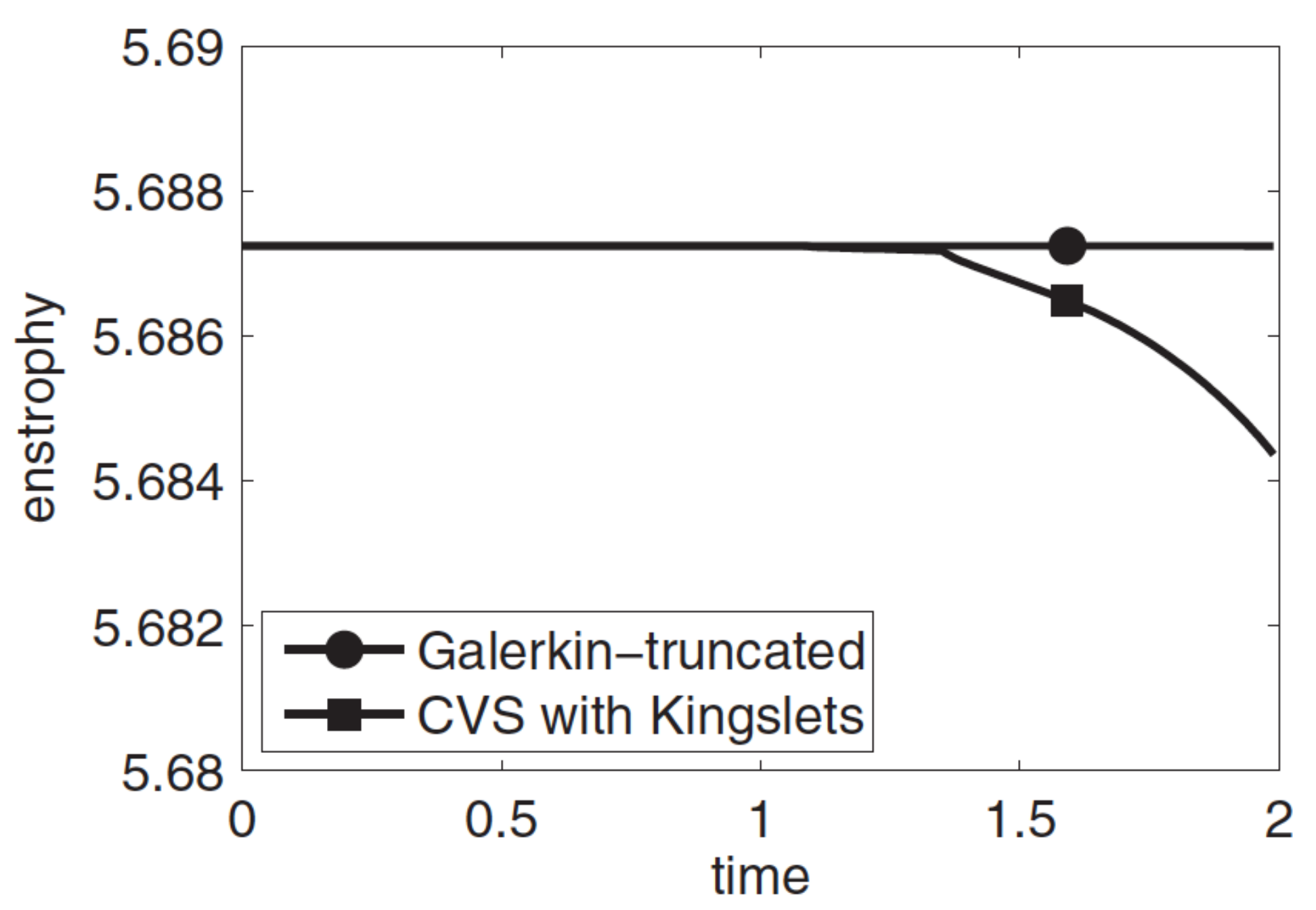}
\end{center}
\caption{Filtering of 2D incompressible Euler using complex-valued wavelets (Kingslets). Evolution of enstrophy $1/2 || \omega ||_2^2$ for the Galerkin truncated case and the adaptive wavelet filtered case using Kingslets. From \cite{PNFS13}.\label{fig:2deuler_enstrophy}}
\end{figure}

\medskip

The three-dimensional Fourier Galerkin computations of incompressible Euler have been performed at resolution $N=512^3$ in a periodic cubic domain with a fourth order Runge-Kutta scheme for time integration \cite{Farge2017}. 
A statistically stationary flow of fully developed homogeneous isotropic turbulence obtained by DNS is used as initial condition.
For CVS filtering Coiflet 12 wavelets \cite{daubechies1992} were used. Note that the wavelet decomposition and subsequent filtering have been applied to the vorticity ${\bm \omega} = \nabla \times {\bm u}$ (and not to the velocity ${\bm u}$) in each time step and subsequently the filtered velocity has been computed by applying the Biot-Savart operator $(\nabla \times)^{-1}$ in Fourier space. 

\begin{figure}
\begin{center}
\includegraphics[width=0.49\textwidth]{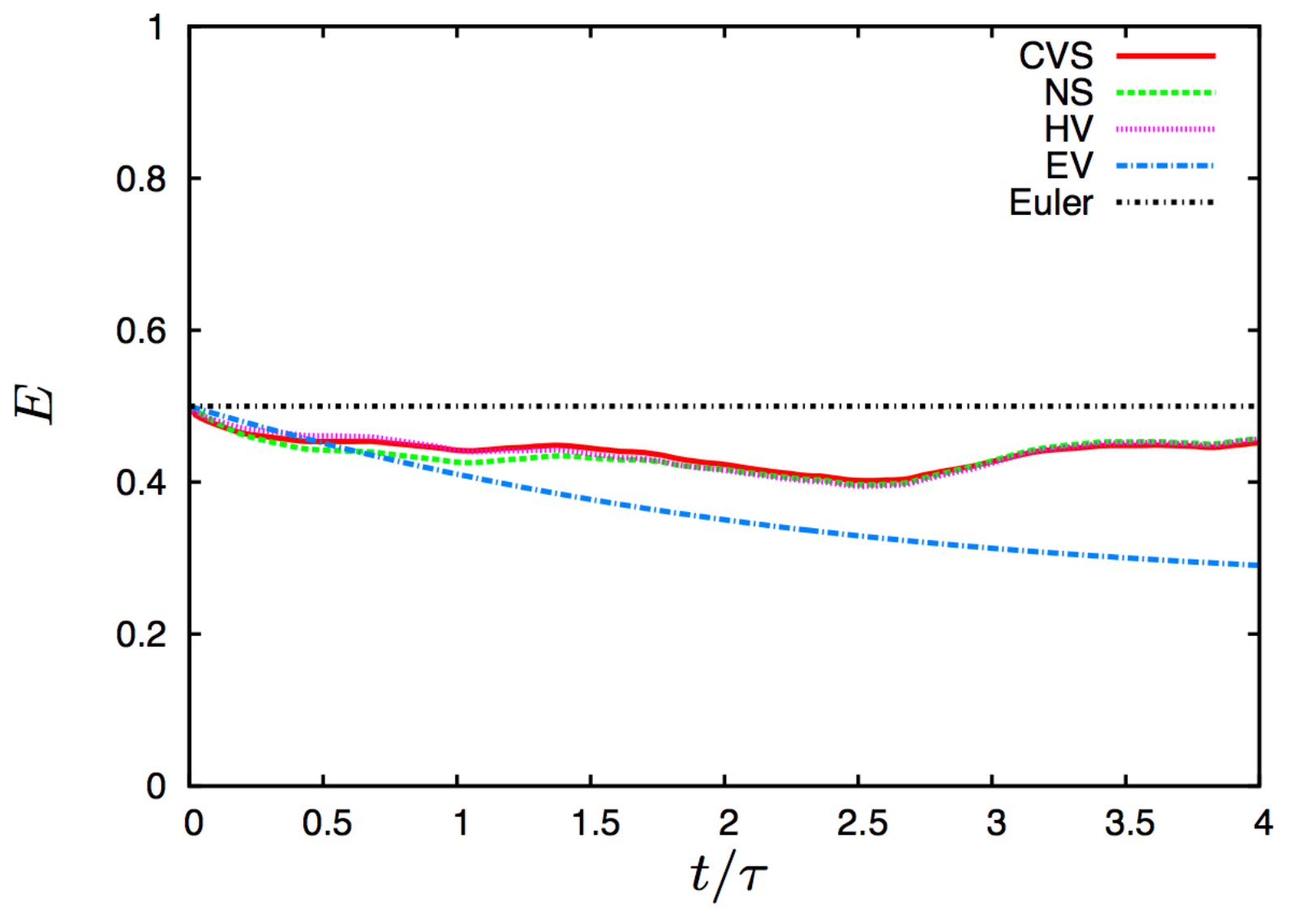}
\includegraphics[width=0.49\textwidth]{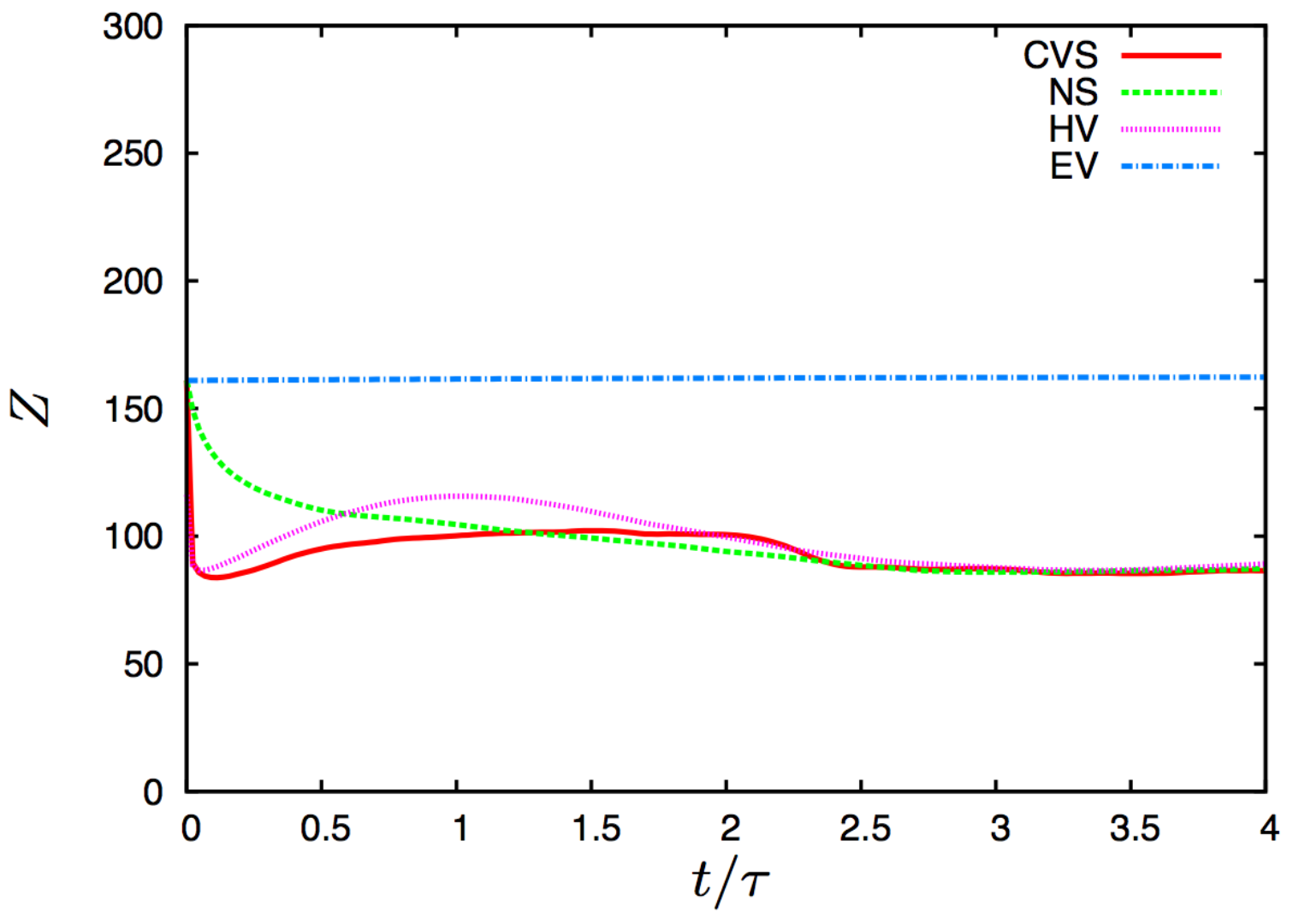}
\end{center}
\caption{Energy (left) and enstrophy (right) evolution for 3D incompressible Euler using for Galerkin truncated Euler (Euler), wavelet filtered Euler (CVS) and Navier-Stokes (NS). HV and EV stand for hyperviscous regularization and EV for Euler-Voigt, respectively, which are not discussed here. From \cite{Farge2017}.\label{fig:3deuler_energy}}
\end{figure}

The time evolution of the energy, $\frac{1}{2} || {\bm u} ||_2^2 $, and enstrophy, $\frac{1}{2} || {\bm \omega} ||_2^2 $, in figure~\ref{fig:3deuler_energy} first shows that the Galerkin truncted Euler computation preserves energy and that enstrophy grows rapidly in time due to the absence of regularization. For CVS we can observe that energy is dissipated, similar to what is observed for Navier-Stokes and that enstrophy also exhibits a similar evolution as NS and does not grow rapidly.

Visualizations of intense vorticity structures in figure~\ref{fig:3deuler_vorticty} for CVS and NS show their similar tube-like character, while the Galerkin truncated Euler solution is similar to Gaussian white noise without the presence of coherent structures.
For details including a physical interpretation of the results we refer to \cite{Farge2017}.

\begin{figure}
\begin{center}
\includegraphics[width=0.31\textwidth]{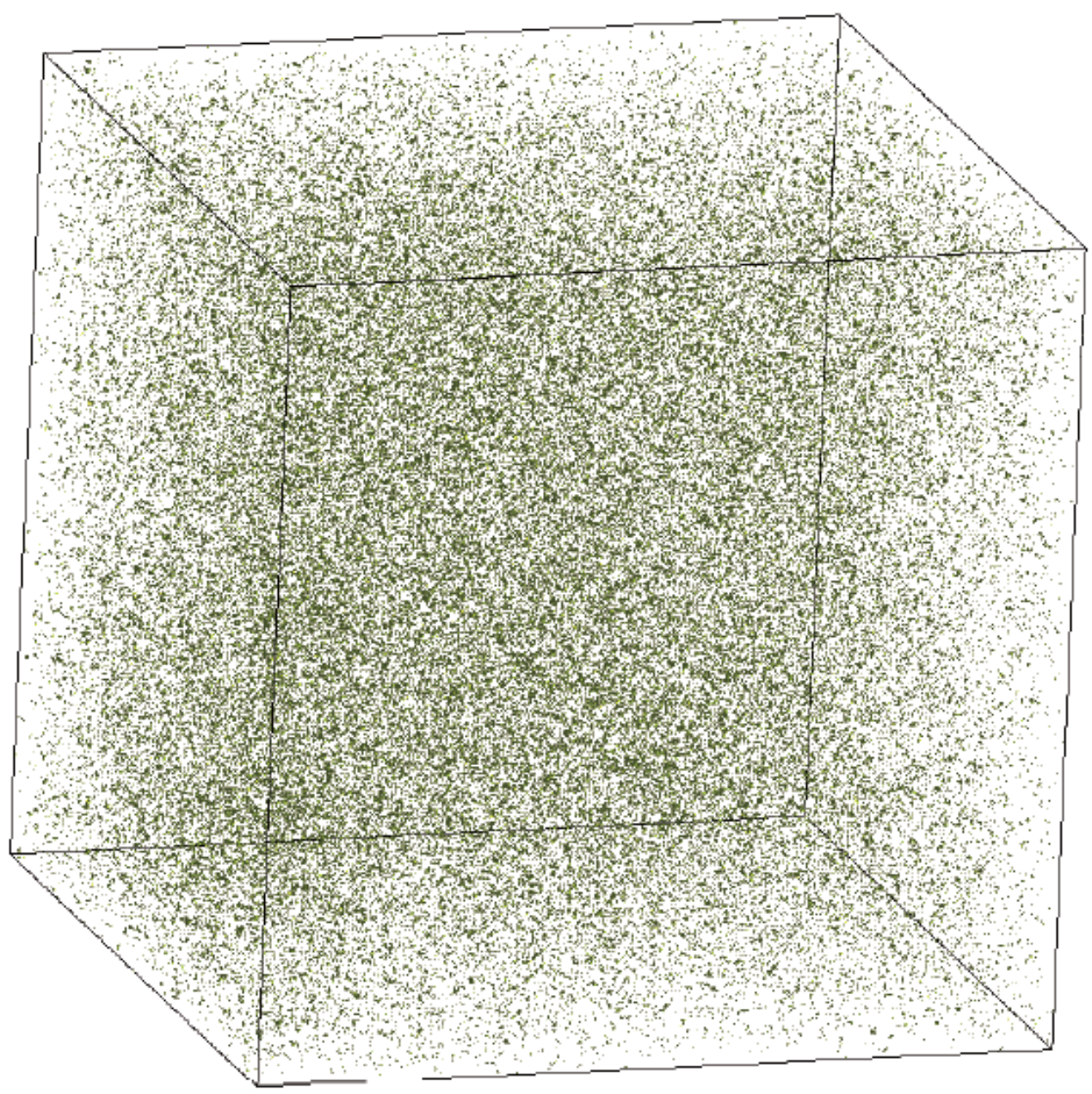}
\includegraphics[width=0.31\textwidth]{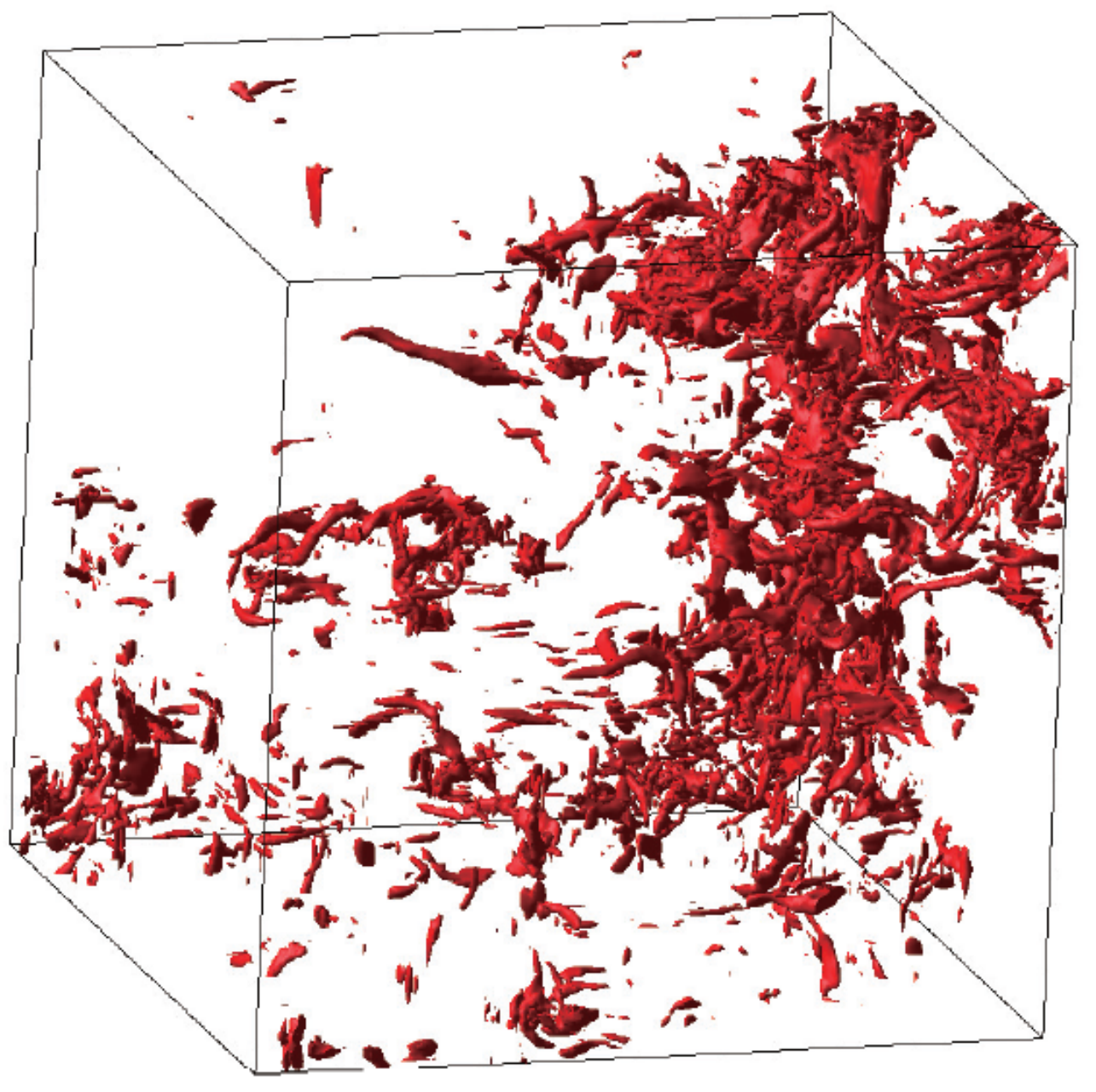}
\includegraphics[width=0.33\textwidth]{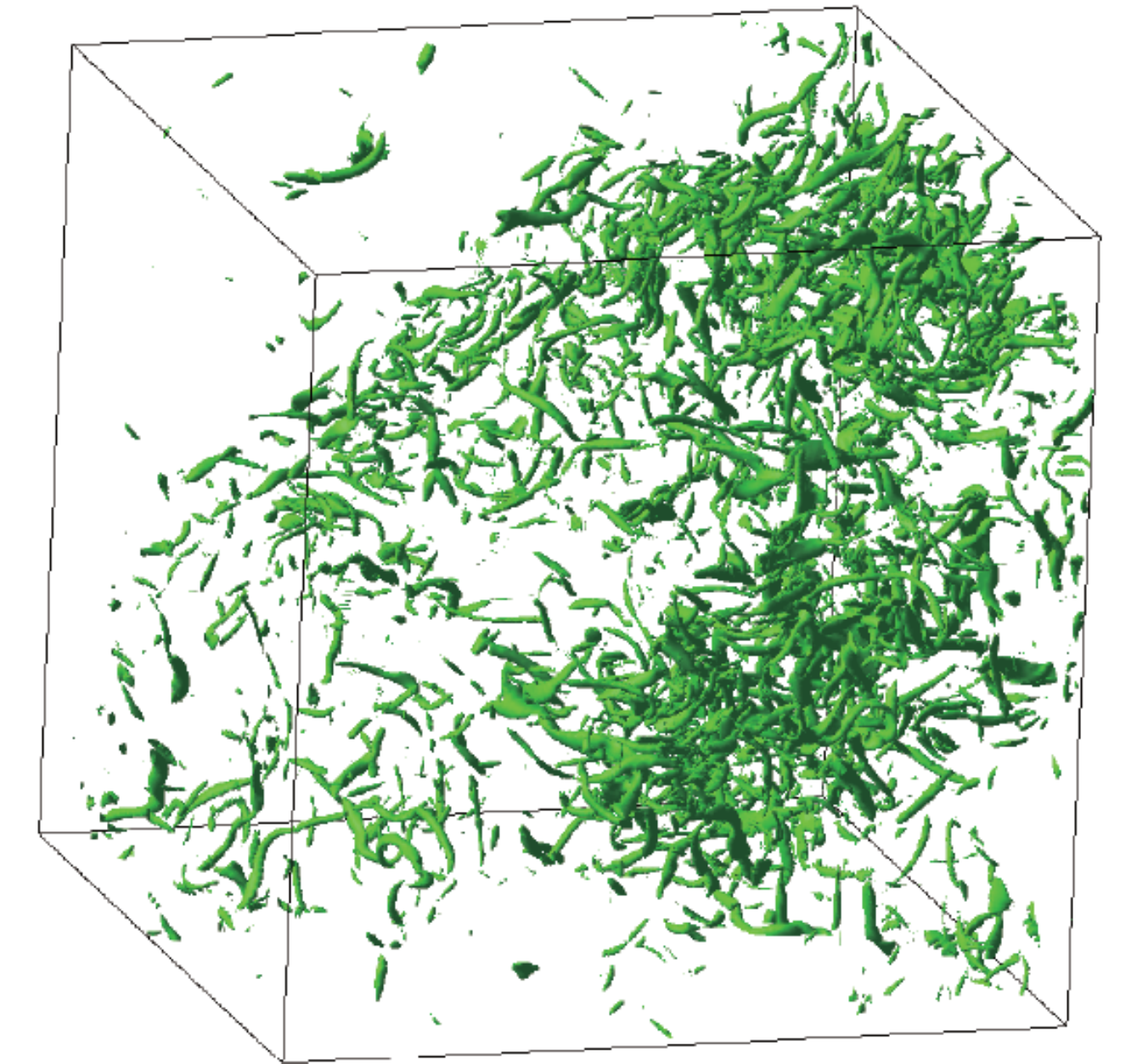}
\end{center}
\caption{Vorticity isosurfaces, $|{\bm \omega}| = M + 4 \sigma$ (where $M$ is the mean value and $\sigma$ the standard deviation of the modulus of vorticity of NS) for 3D incompressible Euler using  Galerkin truncated Euler (Euler, left), wavelet filtered Euler (CVS, center) and Navier-Stokes (NS, right) at time $t/\tau = 3.4$. From \cite{Farge2017}. 
\label{fig:3deuler_vorticty}}
\end{figure}

\section{Conclusions}
\label{sec:concl}

We presented a mathematical framework for analyzing dynamical Galerkin discretizations of evolutionary PDEs. The concept of weak formulations of countable ODEs with non smooth right-hand side in Banach spaces is used.
We showed that changing the set of active basis functions, which implies that the projection operators are non differentiable in time, can introduce energy dissipation. This feature is of crucial interest for adaptive schemes for time dependent equations, e.g., adaptive wavelet schemes for hyperbolic conservation laws and yields a mathematical explanation for their regularizing properties due to dissipation.

Numerical experiments illustrated the above results for the inviscid Burgers equation and the incompressible Euler equations in two and three space dimensions.
To this end the concept of pseudo-adaptive simulations was introduced to test the influence of wavelet thresholding, while solving the PDE with the classical Fourier Galerkin discretization.
The results showed that adaptive wavelet based regularization (i.e., filtering out the weak wavelet coefficients) of Galerkin schemes introduce dissipation together with related space adaptivity. The latter can be used for reducing the computational cost in fully adaptive computations.
Finally, let us mention an interesting link exists with LES models, see e.g., \cite{SZFA2006}, as the equivalence between nonlinear wavelet thresholding (using Haar wavelets) and a single step of explicitly discretized nonlinear diffusion can be shown, see \cite{MWS2003}.

Perspectives of this work are systematic studies of nonlinear hyperbolic conservation laws using adaptive Galerkin discretizations, in particular wavelet-based schemes and their regularization properties.


\begin{thebibliography}{99}
%
\bibitem{azzalini2005}
A. Azzalini, M. Farge and K. Schneider. Nonlinear wavelet thresholding: A recursive method to determine the optimal denoising threshold. {\it Applied and Computational Harmonic Analysis}, 18(2), 177-185, 2005.

\bibitem{BLTi2008} 
C. Bardos, J. S. Linshiz, and E. S. Titi. Global regularity for a Birkhoff-Rott-$\alpha$ approximation of the dynamics of vortex sheets of the 2d Euler equations. {\it Physica D: Nonlinear Phenomena}, 237(14–17):1905–1911, 2008

\bibitem{BaTa13}
C. Bardos and E. Tadmor.
Stability and spectral convergence of Fourier method for nonlinear problems: on the shortcomings of the 2/3 de-aliasing method. {\it Numerische Mathematik}, {\bf 129}(4), 749-782, 2013.

\bibitem{BLSB1981} 
C. Basdevant, B. Legras, R. Sadourny, and M. B\'eland. A study of barotropic model flows: intermittency, waves and predictability. {\it Journal of the Atmospheric Sciences}, 38:2305–2326, 1981

\bibitem{Brandt77}
A. Brandt. Multi-level adaptive solutions to boundary-value problems, {\it Math. Comp.}, 31(1977), pp. 333–390.

\bibitem{CQHZ88}
C. Canuto, A. Quarteroni, M. Y. Hussaini, and T. A. Zang. Spectral methods in fluid dynamics. Springer-Verlag, 1988.

\bibitem{Cohen00}
A. Cohen. Wavelet methods in numerical analysis. Handbook of Numerical
Analysis. Eds. P.G. Ciarlet \& J.L. Lions, Vol. 7, Elsevier, 2000.

\bibitem{daubechies1992}
I. Daubechies. Ten lectures on wavelets. Society for Industrial and Applied Mathematics, Philadelphia, 1992.

\bibitem{Deim77}
K. Deimling. 
{\it Ordinary differential equations in Banach spaces.} Springer, 1977.

\bibitem{ESRF2021}
T. Engels, K. Schneider, J. Reiss and M. Farge. 
A wavelet adaptive method for multiscale simulation of turbulent flows in
flying insects. 
{\it Commun. Comput. Phys.}, 30(4), 1118-1149, 2021.

\bibitem{FSK1999}
M. Farge, K. Schneider and N. Kevlahan.
Non-Gaussianity and coherent vortex simulation for two-dimensional turbulence using an adaptive orthogonal wavelet basis.
{\it Phys. Fluids}, 11(8), 2187–2201, 1999.

\bibitem{Farge2017}
M. Farge, N. Okamoto, K. Schneider and K. Yoshimatsu.
Wavelet-based regularization of the Galerkin truncated three-dimensional incompressible Euler flows.
{\it Phys. Rev. E}, 96, 063119, 2017. 

\bibitem{Fili2013}
A.F. Filippov. Differential equations with discontinuous right hand sides: control systems (Vol. 18). Springer Science \& Business Media, 2013.

\bibitem{Gottlieb2001} 
D. Gottlieb and J. S. Hesthaven. Spectral methods for hyperbolic problems. {\it Journal  of  Computational  and  Applied  Mathematics}, 128(1–2):83–131, 2001.

\bibitem{Harten1983} 
A. Harten. High resolution schemes for hyperbolic conservation laws. {\it Journal  of  Computational  Physics}, 49(3), 357–393, 1983.

\bibitem{IGKa2009}
T. Ishihara, T. Gotoh and Y. Kaneda. Study of high–Reynolds number isotropic turbulence by direct numerical simulation. {\it Annual Review of Fluid Mechanics}, 41, 165-180, 2009

\bibitem{KhTi2008} 
B. Khouider and E. S. Titi. An inviscid regularization for the surface quasi-geostrophic equation. {\it Communications on Pure and Applied Mathematics}, 61(10), 1331–1346, 2008.

\bibitem{Lee1952}
T.D. Lee. On some statistical properties of hydrodynamical and magneto-hydrodynamical fields. {\it Quarterly of Applied Mathematics}, 10(1), 69-74, 1952.

\bibitem{MWS2003}
P. Mr\'azek,  J. Weickert  and  G. Steidl. Correspondences  between  wavelet  shrinkage  and  nonlinear diffusion.  In  L.D.  Griffinand  M.  Lillholm  (Eds.), Scale-Space  2003, LNCS  vol.  2695  (Berlin:  Springer), pp. 101–116, 2003.

\bibitem{MFNBR2020}
S.D. Murugan, U. Frisch, S. Nazarenko, N. Besse and S.S. Ray. Suppressing thermalization and constructing weak solutions in truncated inviscid equations of hydrodynamics: Lessons from the Burgers equation. {\it Physical Review Research}, 2(3), 033202, 2020.

\bibitem{Nguyenvanyen2008}
R. Nguyen van yen, M. Farge, D. Kolomenskiy, K. Schneider and N. Kingsbury. 
Wavelets meet Burgulence: CVS-filtered Burgers equation. 
{\it Physica D: Nonlinear Phenomena}, 237(14), pp.2151-2157, 2008. 

\bibitem{Nguyenvanyen2009}
R. Nguyen van yen, M. Farge and K. Schneider. 
Wavelet regularization of a Fourier-Galerkin method for solving the 2D incompressible Euler equations. 
{\it ESAIM: Proceedings},  {\bf 29}, 89--107, 2009.

\bibitem{OYSFK11}
N. Okamoto, K. Yoshimatsu, K. Schneider, M. Farge and Y. Kaneda. 
Coherent vortex simulation of three-dimensional decaying homogeneous isotropic turbulence.
{\it SIAM Multiscale Model. Simul.}, {\bf 9}(3), 1144-1161, 2011.

\bibitem{Orlandi2000}
P. Orlandi. {\it Fluid Flow Phenomena: A Numerical Toolkit.} Springer, 2000.

\bibitem{Osher1982} 
S.  Osher  and  F.  Solomon.  Upwind  difference  schemes  for  hyperbolic  systems  of  conservation  laws. {\it Mathematics  of Computation}, 38(158):339–374, 1982

\bibitem{Pandit1982}
S.G. Pandit and S.G. Deo.
{\it Differential equations involving impulses.} Lecture Notes in Mathematics, Vol. 954, Springer, 1982.

\bibitem{PNFS13}
R. M. Pereira, R. Nguyen van yen, M. Farge and K. Schneider. 
Wavelet methods to eliminate resonances in the Galerkin-truncated Burgers and Euler equations. 
{\it Phys. Rev. E}, {\bf 87}, 033017, 2013.

\bibitem{Kingslets}
N. Kingsbury. Complex wavelets for shift
invariant analysis and filtering of signals. {\it Appl. Comput. Harm. Anal.}, 10(3):234–253, 2001.

\bibitem{RFNM11}
S. S. Ray, U. Frisch, S. Nazarenko, and T. Matsumuto.
Resonance phenomenon for the Galerkin-truncated Burgers and Euler equations. 
{\it Phys. Rev. E}, {\bf 84}, 016301, 2011.

\bibitem{Vergassola1994}
M. Vergassola, B. Dubrulle, U. Frisch and A. Noullez. 
Burgers' equation, devil's staircases and the mass distribution for large-scale structures. 
{\it Astronomy and Astrophysics,} {\bf 289}, 325--356, 1994.

\bibitem{SZFA2006}
K. Schneider, J. Ziuber, M. Farge and A. Azzalini. Coherent vortex extraction and simulation of 2D isotropic turbulence. {\it Journal of Turbulence}, {\bf 7}, N44, 2006.

\bibitem{ScVa10}
K. Schneider and O. Vasilyev. 
Wavelet methods in computational  fluid dynamics. 
{\it Annu. Rev. Fluid Mech.}, {\bf 42}, 473--503, 2010.

\bibitem{Schwabik1992}
S. Schwabik.
{\it Generalized ordinary differential equations} (Vol. 5). Singapore: World Scientific, 1992.

\bibitem{Sweby1984}
P. K. Sweby. High resolution schemes using flux limiters for hyperbolic conservation laws. {\it SIAM Journal on Numerical Analysis}, 21(5):995–1011, 1984.

\bibitem{Tadmor1989}
E.  Tadmor.  Convergence  of  spectral  methods  for  nonlinear  conservation  laws. {\it SIAM  Journal  on  Numerical  Analysis}, 26(1):30–44, 1989


\end{thebibliography}
\bibliographystyle{siam}

\end{document}